\documentclass{amsart}
\usepackage{amssymb,enumerate,amsmath,amsthm,amsfonts,mathtools} 
\usepackage{graphics, float, pdfpages, subcaption} 
\usepackage[shortlabels]{enumitem}
\usepackage{pgf,hyperref}
\theoremstyle{definition}
\newtheorem{theo}{Theorem}[section]
\newtheorem*{theo*}{Theorem}
\newtheorem{prop}[theo]{Proposition}
\newtheorem{lem}[theo]{Lemma}

\newtheorem{rem}[theo]{Remark}
\newtheorem*{claim}{Claim}
\newtheorem{defi}[theo]{Definition}

\newtheoremstyle{custom}{}{}{\itshape}{}{\bfseries}{.}{.5em}{#1 \thmnote{#3}}
\theoremstyle{custom}
\newtheorem*{ctheo}{Theorem}
\newtheorem*{cprop}{Proposition}

\def\Z{{\mathbb Z}}

\def\N{{\mathbb N}}

\def\cD{{\mathcal D}}

\def\cB{{\mathcal B}}

\def\cA{{\mathcal A}}

\def\cS{{\mathcal S}}

\def\cW{{\mathcal W}}

\makeatletter
\DeclareRobustCommand{\cev}[1]{%
{\mathpalette\do@cev{#1}}%
}
\newcommand{\do@cev}[2]{%
\vbox{\offinterlineskip
	\sbox\z@{$\m@th#1 x$}%
	\ialign{##\cr
		\hidewidth\reflectbox{$\m@th#1\vec{}\mkern4mu$}\hidewidth\cr
		\noalign{\kern-\ht\z@}
		$\m@th#1#2$\cr
	}%
}%
}
\makeatother  

\author{Basti\'an Espinoza}
\address{Departamento de Ingenier\'{\i}a Matem\'atica and Centro de Modelamiento Matem\'atico, Universidad de Chile \& IRL-CNRS 2807, Beauchef 851, Santiago, Chile.}
\address{Laboratoire Amiénois de Mathématique Fondamentale et Appliquée, CNRS UMR 7352, Université de Picardie Jules Verne, 33 rue Saint Leu, 80039 Amiens cedex 1, France.} 
\email{bespinoza@dim.uchile.cl}

\author[add1]{Alejandro Maass}
\address{Departamento de Ingenier\'{\i}a Matem\'atica and Centro de Modelamiento Matem\'atico, Universidad de Chile \& IRL-CNRS 2807, Beauchef 851, Santiago, Chile.}
\email{amaass@dim.uchile.cl}

\thanks{This research was partially supported by grant ANID-AFB 170001. The first author thanks Doctoral Felowship CONICYT-PFCHA/Doctorado Nacional/2020-21202229}

\begin{document}
\title{On the automorphism group of minimal $\mathcal{S}$-adic subshifts of finite alphabet rank}
\date{\today}
\maketitle

\begin{abstract}
It has been recently proved that the automorphism group of a minimal subshift with non-superlinear word complexity is virtually $\mathbb{Z}$ \cite{DDPM15,CK15}. In this article we extend this result to a broader class proving that the automorphism group of a minimal $\mathcal{S}$-adic subshift of finite alphabet rank is virtually $\mathbb{Z}$. The proof is based on a fine combinatorial analysis of the asymptotic classes in this type of subshifts, which we prove are a finite number.
\end{abstract}

\markboth{Basti\'an Espinoza, Alejandro Maass}{On the automorphism group of minimal $\mathcal{S}$-adic subshifts of finite alphabet rank}

\section{Introduction} \label{sec:Introduction}

Let ${\mathcal A}$ be a finite alphabet and let $X\subseteq {\mathcal A}^\Z$ be a \emph{subshift}, \emph{i.e.}, a closed set that is invariant under the left shift $T:{\mathcal A}^\Z \to {\mathcal A}^\Z$. The automorphism group of $(X,T)$, Aut$(X,T)$, is the set of homeomorphisms from $X$ to itself which commute with $T$. The study of the automorphism group of low word complexity subshifts $(X,T)$ has attracted a lot of attention in recent years. By word complexity we mean the increasing function $p_X\colon \N\to \N$ which counts the number of  words of length $n\in \N$  appearing in points of the subshift $(X,T)$. In contrast to the case of non trivial mixing shifts of finite type or synchronized systems, where the algebraic structure of this group can be very rich 
\cite{BLR,KR90,FF}, the automorphism group of low word complexity subshifts is expected to present high degrees of rigidity. The most relevant example illustrating this fact are minimal subshifts of non-superlinear word complexity, where the  automorphism group is virtually $\Z$ \cite{CK15,DDPM15}. Interestingly, in \cite{Salo} (and then in \cite{DDPM15} in a more general class) the author provides a Toeplitz subshift with complexity $p_X(n)\leq Cn^{1.757}$, whose automorphism group is not finitely generated. So some richness in the algebraic structure of the automorphism groups of  low word complexity subshifts can arise. Other low word complexity subshifts have been considered by Cyr and Kra in a series of works. In \cite{CK16-1} they proved that for transitive subshifts, if 
$\displaystyle\liminf_{n \to +\infty} p_X(n)/n^2=0$, then the quotient ${\rm Aut}(X,T)/\langle T \rangle$ is a periodic group, where $\langle T \rangle$ is the group spanned by the shift map; and in \cite{CK16-2} for a large class of minimal subshifts of subexponential complexity they also proved that the automorphism group is amenable.  All these classes and examples show that there is still a lot to be understood on the automorphism groups of low word complexity subshifts.

In this article we study the automorphism group of \emph{minimal $\cS$-adic subshifts of finite or bounded alphabet rank}. This class of minimal subshifts is somehow the most natural class containing minimal subshifts of non-superlinear complexity, but it is much broader as was shown in \cite{DDPM15,DDPM20}.
Moreover, this class contains several well studied minimal symbolic systems. Among them, substitution subshifts, linearly recurrent subshifts, symbolic codings of interval exchange transformations, dendric subshifts and some Toeplitz sequences. Thus, this class represents a useful framework for both, proving general theorems in the low word complexity world and building subshifts with interesting dynamical behavior. The descriptions made in \cite{BKMS10} of its invariant measures and in \cite{DFM15} of its eigenvalues are examples of the former, and the well behaved $\cS$-adic codings of high dimensional torus translations from \cite{BST20} is an example of the later.
\smallskip

The main result of this article is the following rigidity theorem:
\begin{theo}\label{theo:main}
Let $(X,T)$ be a minimal $\cS$-adic subshift given by an everywhere growing directive sequence $\boldsymbol{\tau} = (\tau_n\colon \cA_{n+1}^+\to \cA_n^+)_{n\geq0}$. Suppose that $\boldsymbol{\tau}$ is of finite alphabet rank, \textit{i.e.}, $\displaystyle \liminf_{n\to +\infty} \# \cA_n < +\infty$. Then, Aut$(X,T)$ is virtually $\Z$.
\end{theo}
\smallskip

A minimal $\cS$-adic subshift of \emph{finite topological rank}, as stated in \cite{DDPM20}, is defined as an $\cS$-adic subshift in which the defining directive sequence $\boldsymbol{\tau}$ is proper, primitive, recognizable and with finite alphabet rank. In particular, $\boldsymbol{\tau}$ is everywhere growing. Therefore, Theorem \ref{theo:main} includes all minimal $\cS$-adic subshifts of finite topological rank. Also, in the same paper, the authors prove that minimal subshifts of non-superlinear word complexity are $\cS$-adic of finite topological rank.  Thus, Theorem \ref{theo:main} can be seen as a generalization to a much broader class of the already mentioned results from \cite{CK15} and \cite{DDPM15}. Finally, by results stated in \cite{DDPM20}, Theorem \ref{theo:main} also applies to all level subshifts of minimal Bratteli-Vershik systems of finite topological rank and its symbolic factors.

The proof of Theorem \ref{theo:main} follows from a fine combinatorial analysis of asymptotic classes of $\cS$-adic subshifts of finite alphabet rank. This idea already appeared in \cite{DDPM15}, where the authors prove that the automorphism group of a minimal system is virtually $\Z$ whenever it has finitely many asymptotic classes. The following theorem summarizes this combinatorial analysis. 

\begin{theo}\label{prop:main_technical}
Let $\cW \subseteq \cA^+$ be a set of nonempty words and define $\displaystyle \langle\cW\rangle \coloneqq \min_{w\in\cW}\mathrm{length}(w)$. Then, there exists $\cB \subseteq \cA^{\langle\cW\rangle}$ with
$\#\cB \leq 122(\# \cW)^7$ such that:
if $x,x' \in \cA^\mathbb{Z}$ are factorizable over $\cW$, $x_{(-\infty,0)} = x'_{(-\infty,0)}$ and $x_0\not= x'_0$, then $x_{[-\langle \cW\rangle,0)} \in \cB$.
\end{theo}

Here, the important point is that, despite the fact that the length of the elements in $\cB$ is $\langle\cW\rangle$, the cardinality of $\cB$ depends only on $\# \cW$, and not on $\langle \cW\rangle$. 

Finally, we get a bound for the asymptotic classes of an $\cS$-adic subshift of finite alphabet rank. This result does not require minimality. 

\begin{theo}\label{theo:s_adic_finite_asymptotic}
Let $(X,T)$ be an $\cS$-adic subshift (not necessarily minimal) given by an everywhere growing directive sequence of finite alphabet rank $K$. Then, $(X,T)$ has at most $122K^7$ asymptotic classes.
\end{theo}

\subsection{Organization} 
The paper is organized as follows. In the next section we give some background in topological and symbolic dynamics. In Section \ref{sec:interpretations} we introduce some special ingredients allowing to prove the main theorems: the notions of \emph{interpretation} and \emph{reducibility} of sets of words together with its  properties and the key Proposition \ref{lem:bound_on_irreducible}, whose technical proof is given in Section \ref{sec:main_proof}. 
In Section \ref{sec:main_results} we restate our main results and provide complete proofs. 

\section{Background in topological and symbolic dynamics}
\label{sec:Basics}

All the intervals we will consider consist of integer numbers, \textit{i.e.}, $[a,b] = \{k\in\Z:a\leq k\leq b\}$ with $\ a,b \in \Z$. For us, the set of natural numbers starts with zero, \textit{i.e.}, $\N=\{0,1,\dots\}$.

\subsection{Basics in topological dynamics}\label{subsec:topdyn}
A {\em topological dynamical system} (or just a system) is a pair $(X,T)$, where $X$ is a compact metric space and  $T\colon X \to X$ is a {\em homeomorphism} of $X$. The orbit of $x \in X$ is the set $\{T^nx : n \in \Z\}$. A point $x\in X$ is {\em periodic} if its orbit is a finite set and {\em aperiodic} otherwise. A topological dynamical system is {\em aperiodic} if any point $x\in X$ is aperiodic and is {\em minimal} if the orbit of every point is dense in $X$. We use the letter $T$ to denote the action of a topological dynamical system independently of the base set $X$.

An {\em automorphism} of the topological dynamical system $(X,T)$ is a homeomorphism $\varphi\colon X\to X$ such that $\varphi\circ T=T\circ \varphi$. We use the notation $\varphi\colon (X,T)\to (X,T)$ to indicate the automorphism. The set of all automorphisms of $(X,T)$ is denoted by $\mathrm{Aut}(X,T)$ and is called the {\em automorphism group} of $(X,T)$. It has a group structure given by the composition of functions. It is said that $\mathrm{Aut}(X,T)$ is \textit{virtually $\mathbb{Z}$} if the quotient $\mathrm{Aut}(X,T)/\langle T\rangle$ is finite, where $\langle T\rangle$ is the subgroup generated by $T$.

\subsection{Basics in symbolic dynamics}\label{subsec:symbdyn}
\subsubsection{Words and subshifts} \label{subsec:subsubshifts}

Let ${\mathcal A}$ be a finite set that we call {\em alphabet}. Elements in ${\mathcal A}$ are called {\em letters} or {\em symbols}. The set of finite sequences or {\em words} of length $\ell\in \N$ with letters in $\mathcal A$ is denoted by ${\mathcal A}^\ell$, the set of onesided sequences $(x_n)_{n\in \mathbb{N}}$  in ${\mathcal A}$ is denoted  by ${\mathcal A}^{\mathbb N}$
and the set of twosided sequences $(x_n)_{n\in \mathbb{Z}}$  in ${\mathcal A}$ is denoted by ${\mathcal A}^{\mathbb Z}$. 
Also, a word $w= w_1 \cdots w_{\ell} \in  {\mathcal A}^\ell$, with $\ell > 0$, can be seen as an element of the free monoid ${\mathcal A}^*$ endowed with the operation of concatenation (and whose neutral element is $1$, the empty word), and as an element of the free semigroup $\cA^+ \coloneqq \cA^*\setminus\{1\}$ of nonempty words. The integer $\ell$ is the {\em length} of $w$ and is denoted by $|w|=\ell$; the length of the empty word is $0$.

We write $\leq_p$ and $\leq_s$ for the relations in $\cA^*$ of being prefix and suffix, respectively. We also write $u <_p v$ (resp. $u <_s v$) when $u \leq_p v$ (resp. $u \leq_s v$) and $u \not= v$. When $v = sut$, we say that $u$ {\em occurs} in $v$ or that $u$ is a \emph{subword} of $v$. We also use these notions and notations when considering prefixes, suffixes and subwords of infinite sequences.

Let $\cW \subseteq \cA^*$ be a set of words and $u \in \cA^*$. We write $u\cW = \{uw : w \in \cW\}$, $\cW u = \{wu : w \in \cW\}$, and also
\begin{equation*}
\langle \cW\rangle \coloneqq \min_{w\in \cW}|w|
\qquad\text{and}\qquad
| \cW| \coloneqq \max_{w\in \cW}|w|.
\end{equation*}

The {\em shift map} $T \colon {\mathcal A}^{\mathbb Z} \to {\mathcal A}^{\mathbb Z}$ is defined by $T ((x_n)_{n\in \mathbb{Z}}) = (x_{n+1})_{n\in \mathbb{Z}}$. A {\em subshift} is a topological dynamical system $(X,T)$, where $X$ is a closed and $T$-invariant subset of ${\mathcal A}^{\mathbb Z}$ (we consider the product topology in ${\mathcal A}^{\mathbb Z}$) and $T$ is the shift map. Classically one identifies $(X,T)$ with $X$, so one says that $X$ itself is a subshift. When we say that a sequence  in a subshift is aperiodic, we implicitly mean that this sequence is aperiodic for the action of the shift.   

\subsubsection{Morphisms and substitutions}\label{subsubsec:morphisms}

Let $\cA$ and $\cB$ be finite alphabets and $\tau\colon\cA^+\to\cB^+$ be a morphism between the free semigroups that they define. Then, $\tau$ extends naturally to maps from $\cA^\mathbb{N}$ to itself and from $\cA^\mathbb{Z}$ to itself in the obvious way by concatenation (in the case of a twosided sequence we apply $\tau$ to positive and negative coordinates separately and we concatenate the results at coordinate zero).
We say that $\tau$ is {\em primitive} if for every $a\in\cA$, all letters $b \in \cB$ occur in $\tau(a)$. The minimum length of $\tau$ is the number
\begin{equation*}
\langle\tau\rangle \coloneqq \langle\tau(\cA)\rangle = \min_{a\in\cA}|\tau(a)|.
\end{equation*}

We observe that any map $\tau\colon \cA\to \cB^+$ can be naturally extended to a morphism (that we also denote by $\tau$) from $\cA^+$ to $\cB^+$ by concatenation, and we use this convention throughout the document. So, from now on, all maps between finite alphabets are considered to be morphisms between their associated free semigroups. 

\subsubsection{$\cS$-adic subshifts}\label{subsubsec:Sadicsubshifts}

We recall the definition of an {\em $\cS$-adic subshift} as stated in \cite{BSTY17}.
A {\em directive sequence}  $\boldsymbol{\tau} = (\tau_n \colon \mathcal{A}_{n+1}^+\to \mathcal{A}_n^+)_{n \geq 0}$ is a sequence of morphisms.

For $0\leq n<N$, we denote by $\tau_{[n,N)}$, the morphism $\tau_n \circ \tau_{n+1} \circ \dots \circ \tau_{N-1}$. We say $\boldsymbol{\tau}$ is {\em everywhere growing} if
\begin{equation*}\label{eq:everywhere_growing}
\lim_{N\to +\infty}\langle\tau_{[0,N)}\rangle = +\infty.
\end{equation*}
We say $\boldsymbol{\tau}$ is {\em primitive} if for any $n\in \N$ there exists $N>n$ such that $\tau_{[n,N)}$ is primitive. Observe that $\boldsymbol{\tau}$ is everywhere growing whenever $\boldsymbol{\tau}$ is primitive.

For $n\in \N$, we define
\begin{equation*}
X^{(n)}_{\boldsymbol{\tau}} = \big\{x \in \mathcal{A}_n^\Z :\
\mbox{$ \forall k\in\N$, $x_{[-k,k]}$ occurs in $\tau_{[n,N)}(a)$ for some $N>n$ and $a \in\mathcal{A}_N$}\big\}.
\end{equation*}
These sets clearly define  subshifts. The set $X_{\boldsymbol{\tau}} = X^{(0)}_{\boldsymbol{\tau}}$ is called the \emph{$\cS$-adic subshift} generated by $\boldsymbol{\tau}$ and  $X^{(n)}_{\boldsymbol{\tau}}$ is called the {\em $n$th level of the $\cS$-adic subshift generated by $\boldsymbol{\tau}$}.  
If $\boldsymbol{\tau}$ is everywhere growing, then every $X^{(n)}_{\boldsymbol{\tau}}$, $n\in\N$, is nonempty; if $\boldsymbol{\tau}$ is primitive, then $X^{(n)}_{\boldsymbol{\tau}}$ is minimal and nonempty for every $n\in\N$. There are nonprimitive directive sequences that generate minimal subshifts.

The relation between levels of an $\cS$-adic subshift is given by the following lemma.

\begin{lem}[\cite{BSTY17}, Lemma 4.2]\label{lem:desubstitution}
Let $\boldsymbol{\tau} = (\tau_n \colon \mathcal{A}_{n+1}^+\to \mathcal{A}_n^+)_{n \geq 0}$ be a directive sequence of morphisms. If $0\leq n < N$ and $x \in X_{\boldsymbol{\tau}}^{(n)}$, then there exist $y \in X_{\boldsymbol{\tau}}^{(N)}$ and $k \in \Z$ such that $x = T^k\tau_{[n,N)}(y)$.
\end{lem}

We define the {\em alphabet rank} of a directive sequence $\boldsymbol{\tau}$ as 
$$AR(\boldsymbol{\tau}) = \liminf_{n\to +\infty} \#  \cA_n.$$

In this paper we will deal with systems 
$(X_{\boldsymbol{\tau}},T)$ given by 
an everywhere growing directive sequence 
$\boldsymbol{\tau}$ of finite alphabet rank. This kind of systems generalises the class of \emph{finite topological rank} systems stated for minimal Bratteli-Vershik systems and its symbolic factors (see for example \cite{DFM15}), but is somehow more natural and includes a broader spectrum of systems, not all minimal. It is worth mentioning that finite topological rank minimal systems are either subshifts or odometers \cite{DM}.

A \textit{contraction} of $\boldsymbol{\tau}= (\tau_n \colon \mathcal{A}_{n+1}^+\to \mathcal{A}_n^+)_{n \geq 0}$ is a sequence $\tilde{\boldsymbol{\tau}} = (\tau_{[n_k,n_{k+1})}\colon \cA_{n_{k+1}}^+\to \cA_{n_k}^+)_{k\geq0}$, where $0 = n_0 < n_1 < n_2 < \dots$. Observe that any contraction of $\boldsymbol{\tau}$ generates the same $\cS$-adic subshift $X_{\boldsymbol{\tau}}$. When $\boldsymbol{\tau}$ has finite alphabet rank, there exists a contraction $\tilde{\boldsymbol{\tau}} = (\tau_{[n_k,n_{k+1})}\colon \cA_{n_{k+1}}^+\to \cA_{n_k}^+)_{k\geq0}$ of $\boldsymbol{\tau}$ in which $\cA_{n_k}$ has cardinality $AR(\boldsymbol{\tau})$ for every $k\geq1$.

\section{Notion of {\em Interpretation}}\label{sec:interpretations}

In this section we introduce the concepts of \emph{interpretation} and \emph{double interpretation} of a word together with its basic properties. The definitions we provide here are variants of the same notion used seldom in combinatorics of words, see for example \cite{Lothaire}. The key Proposition \ref {lem:bound_on_irreducible}, where we provide a fundamental upper bound for the number of \emph{irreducible sets of simple double interpretations}, is announced here and proved in the last section of the article.

For the rest of this section we fix an alphabet $\cA$ and a finite set of nonempty words $\cW \subseteq \cA^+$. If $u,v,w \in \cA^*$ are such that $w = uv$, then we write $u = wv^{-1}$ and $v = u^{-1}w$.

\subsection{Interpretations and simple double interpretations}
\label{subsec:interpretations}

\begin{defi}\label{def:interpretation}
Let $d \in \cA^+$. A {\em $\cW$-interpretation} of $d$ is a sequence of words
$I =\mathsf{d}_L, \mathsf{d}_M, \mathsf{d}_R, a$ such that:
\begin{enumerate}
	\item $\mathsf{d}_M \in \cW^*$ and $a \in \cA$; 
	\item there exist $\mathsf{u}_L, \mathsf{u}_R \in \cW$ such that $1\not= \mathsf{d}_L \leq_s \mathsf{u}_L$, $\mathsf{d}_R a \leq_p \mathsf{u}_R$;
	\item $d = \mathsf{d}_L \mathsf{d}_M \mathsf{d}_R$.
\end{enumerate}
See Figure \ref{fig:defi_interpretation} for an illustration of this definition. Note that $\mathsf{d}_M$ and $\mathsf{d}_R$ can be the empty word. The extra letter $a$ will be crucial to handle asymptotic pairs and $\cW$-interpretations later.
\end{defi}

\begin{figure}[H]
\makebox[\textwidth][c]{\includegraphics[scale=1]{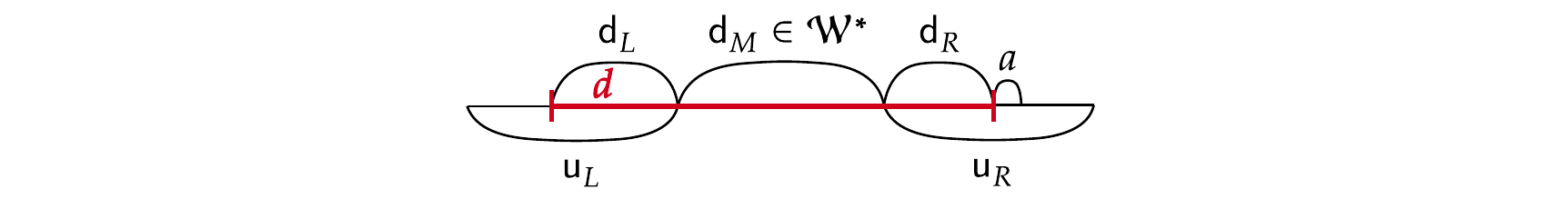}}
\caption{\label{fig:defi_interpretation} Diagram of the $\cW$-interpretation $I = \mathsf{d}_L,\mathsf{d}_M,\mathsf{d}_R,a$ of $d$ in Definition \ref{def:interpretation}.}
\end{figure}

If the context is clear, we will say {\em interpretation} instead of $\cW$-interpretation.
\medskip

Now we make an observation that will be useful when we want to inherit interpretations of a given word to some of its subwords. We state it as a lemma without proof.

\begin{lem}\label{lem:inherit}
Let $I = \mathsf{d}_L, \mathsf{d}_M, \mathsf{d}_R, a$ be a $\cW$-interpretation of $d \in \cA^+$. Suppose that $d' \leq_p d$ satisfy $|d'| \geq |\mathsf{d}_L|$. Then, $d'$ has a $\cW$-interpretation of the form $I' = \mathsf{d}_L, \mathsf{d}'_M, \mathsf{d}'_R, a'$ such that $d' a' \leq_p d a$.
\end{lem}

The proofs of  our main theorems are based in a procedure allowing to reduce the so called \emph{double interpretations} (defined below) to a special class called \emph{simple double interpretations}. 

\begin{defi}\label{defi:double_interpretation}
Let $d\in \cA^+$. A $\cW$-double interpretation (written for short $\cW$-{\em d.i.}) of $d$ is a tuple $D = (I; I')$, where $I = \mathsf{d}_L,\mathsf{d}_M,\mathsf{d}_R,a$, $I'= \mathsf{d}'_L,\mathsf{d}'_M,\mathsf{d}'_R,a'$ are $\cW$-interpretations of $d$ such that $a\not=a'$. We say that $D$ is \textit{simple} if in addition
\begin{enumerate}
\item\label{defi:double_interpretation:1} $\mathsf{d}'_M \mathsf{d}'_R \leq_s \mathsf{d}_R$, and 
\item\label{defi:double_interpretation:2} $\mathsf{d}'_L \in \cW$ or $|\mathsf{d}'_L| \geq |\mathsf{u}|$ for some $\mathsf{u} \in \cW$ having $\mathsf{d}_R a$ as a prefix.
\end{enumerate}
\end{defi}
Again, if there is no ambiguity, we will omit $\cW$ and simply say {\em double interpretation} or  {\em d.i.}

Note that if $D$ is simple, then $D' = (I';I)$ is a d.i., which is not necessarily simple. Condition (1) in the previous definition says that $\mathsf{d}'_L$, the left-most word of $I'$, ``touches'' $\mathsf{d}_R$, the right-most word of $I$; see Figure \ref{fig:defi_di} for an illustration of this. Condition (2) is more technical and we will comment about it at the end of the Subsection \ref{sec:reducibility}. 

\begin{rem}\label{rem:di_basic_props}
From condition \eqref{defi:double_interpretation:2} in previous definition we have that $|\mathsf{d}'_L|, |d| \geq \langle \cW\rangle$, whenever $D$ is a simple $\cW$-d.i.
\end{rem}

\begin{figure}[H]
\makebox[\textwidth][c]{\includegraphics[scale=1]{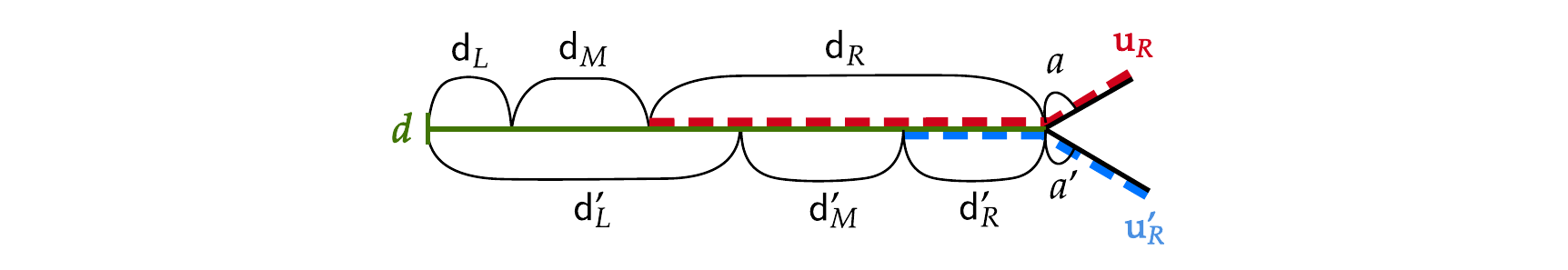}}
\caption{\label{fig:defi_di} Diagram of a d.i. of $d$ satisfying (1) in Definition \ref{defi:double_interpretation}. Here, $\mathsf{d}_R a \leq_p \mathsf{u}_R$ and $\mathsf{d}'_R a' \leq_p \mathsf{u}'_R$, where
$\mathsf{u}_R, \mathsf{u}'_R$ are the words given in condition (2) of Definition \ref{def:interpretation}.}
\end{figure}

The next lemma will be useful to build a \emph{simple} double interpretation from a word having a double interpretation.

\begin{lem}\label{lem:extract_di}
Let $D = (I=\mathsf{d}_L,\mathsf{d}_M,\mathsf{d}_R,a; I'=\mathsf{d}'_L, \mathsf{d}'_M, \mathsf{d}'_R, a')$ be a double interpretation of a word $d \in \cA^+$. Suppose that $\mathsf{d}'_L \in \cW$ and $|\mathsf{d}_L| \leq |\mathsf{d}'_L\mathsf{d}'_M|$. Then, there exists $e \leq_s d$ with a simple double interpretation.
\end{lem}
\begin{proof}
By considering the shortest suffix of $d$ verifying the hypotheses of the lemma we can assume without loss of generality that this suffix is $d$ itself. We consider three cases. 

\begin{enumerate}[wide, labelwidth=!, labelindent=0pt]
\item $\mathsf{d}'_L <_p \mathsf{d}_L$. This condition and the hypotheses of the lemma imply that $\mathsf{d}'_L <_p \mathsf{d}_L \leq_p \mathsf{d}'_L \mathsf{d}'_M$. Therefore, $\mathsf{d}'_M$ is not the empty word and we can write $\mathsf{d}'_M = uv$, with $u \in \cW$ and $v\in \cW^*$.
Then, $e\coloneqq \mathsf{d}'_M \mathsf{d}'_R <_s d$ has the interpretations $J = (\mathsf{d}'_L)^{-1}\mathsf{d}_L, \mathsf{d}_M, \mathsf{d}_R, a$ (here we are using that $(\mathsf{d}'_L)^{-1}\mathsf{d}_L \not= 1$) and $J' = u, v, \mathsf{d}'_R, a'$. But $u \in \cW$ and $|(\mathsf{d}'_L)^{-1}\mathsf{d}_L| \leq |(\mathsf{d}'_L)^{-1}\mathsf{d}'_L\mathsf{d}'_M| = |uv|$, so $e$ is a strict suffix of $d$ having a d.i. $E \coloneqq (J; J')$ verifying the hypotheses of the lemma, which contradicts the minimality of $d$. Thus, this case is incompatible with the hypotheses.  
\item $\mathsf{d}_L <_p \mathsf{d}'_L$. 
If $D$ is not a simple d.i. we have 
$\mathsf{d}_R <_s \mathsf{d}'_M \mathsf{d}'_R$ since $\mathsf{d}'_L \in \cW$ and then $\mathsf{d}_L <_p \mathsf{d}'_L \leq_p \mathsf{d}_L \mathsf{d}_M$. This implies that $\mathsf{d}_M$ is not the empty word. 
Then, we can write $\mathsf{d}_M = uv$ with $u\in \cW$ and $v\in\cW^*$. We have that 
$E = (J=\mathsf{d}_L^{-1}\mathsf{d}'_L, \mathsf{d}'_M, \mathsf{d}'_R, a'; \ J'=u,v, \mathsf{d}_R, a)$ is a d.i. of $e \coloneqq \mathsf{d}_M\mathsf{d}_R <_s d$ which, in addition, satisfies $u \in \cW$ and $|\mathsf{d}_L^{-1}\mathsf{d}'_L| \leq |uv|$. This contradicts the minimality of $d$ and $D$ must be simple. 
\item $\mathsf{d}_L = \mathsf{d}'_L$. 
If $\mathsf{d}_M = 1$ or 
$\mathsf{d}'_M = 1$, it follows directly from definition that $D=(I,I')$ or $D'=(I',I)$ are simple d.i. respectively. So we assume 
$\mathsf{d}_M \not= 1$ and $\mathsf{d}'_M \not= 1$. Therefore, we can write $\mathsf{d}_M = uv$ and $\mathsf{d}'_M = u'v'$, with $u,u' \in \cW$ and $v,v'\in\cW^*$. Let $e \coloneqq \mathsf{d}_M \mathsf{d}_R = \mathsf{d}'_M \mathsf{d'}_R$, $J = u,v,\mathsf{d}_R,a$ and $J' = u',v',\mathsf{d}'_R,a'$. Observe that when $|u'| \leq |u|$, $E = (J'; J)$ is a d.i. of $e$ satisfying $u \in \cW$ and $|u'| \leq |uv|$, and when $|u| \leq |u'|$, $E = (J;J')$ is a d.i. of $e$ satisfying $u' \in \cW$ and $|u| \leq |u'v'|$. In both cases we get a contradiction with the minimality of $d$. Then, in this case either $D$ or $D'$ is a simple d.i. of $d$.
\end{enumerate}
\end{proof}

A point $x \in \cA^\Z$ is {\em factorizable} over $\cW$ if there exist a point $y \in \cW^\Z$ and $k \in \Z$ such that $x_{[k,\infty)} = y_0y_1y_2\cdots$ and $x_{(-\infty,k)} = \cdots y_{-3}y_{-2}y_{-1}$. For example, if $\boldsymbol{\tau}$ is a directive sequence, $0\leq n < N$ and $x \in X_{\boldsymbol{\tau}}^{(n)}$, from Lemma \ref{lem:desubstitution} we see that $x$ is factorizable over $\tau_{[n,N)}(\cA_{N})$.

The last lemma of this subsection gives the relation between asymptotic pairs that are factorizable over the set of words $\cW$ and simple double interpretations over $\cW$. This lemma is crucial to reduce our combinatorial studies in next sections to the case of simple double interpretations.

\begin{lem}\label{lem:asymptotic_to_di}
If $x,x' \in \cA^\Z$ are factorizable over  $\cW$, $x_{(-\infty,0)} = x'_{(-\infty,0)}$ and $x_0 \not= x'_0$, then there exists a word $e \leq_s x_{(-\infty,0)}$ having a simple double interpretation over $\cW$.
\end{lem}
\begin{proof}
Let $l \geq 2|\cW|$ and $d \coloneqq x_{[-l,0)}$. Then $d$ inherits in a natural way interpretations $I = \mathsf{d}_L,\mathsf{d}_M,\mathsf{d}_R,a$ and $I'=\mathsf{d}'_L,\mathsf{d}'_M,\mathsf{d}_R',a'$ from the factorizations of $x$ and $x'$ respectively. Since $a = x_0 \not= x'_0 = a'$, the tuple $D \coloneqq (I;I')$ is a d.i. Moreover, by choosing adequately $l$ we can suppose that $d'_L \in \cW$. Also, $|\mathsf{d}_L| \leq |\cW| \leq l-|\mathsf{d}'_R| = |\mathsf{d} '_L\mathsf{d}'_M|$, so the hypotheses of Lemma \ref{lem:extract_di} hold. Thus $d$ (and of course $x_{(-\infty,0)}$) has a suffix $e$ with a simple double interpretation over $\cW$. This proves the lemma. 
\end{proof}

\subsection{Reducible and irreducible simple double interpretations}
\label{sec:reducibility}

In this section we introduce the notions of \textit{reducible} and \textit{irreducible} sets of simple double interpretations.
In Proposition \ref{lem:bound_on_irreducible} we provide an upper bound for the size of irreducible sets of simple d.i. (the proof of this proposition is very technical and is postponed until Section \ref{sec:main_proof}). Thus, even if in some cases it is not necessary, most of the notions appearing in this section will be considered only for simple d.i.

For the rest of the paper each time we use a letter $D$ to denote a d.i. on $\cW$, then it double interprets the word $d \in \cA^+$ and is written $D = (I_D=\mathsf{d}_L,\mathsf{d}_M,\mathsf{d}_R,a_D;\ I'_D=\mathsf{d}'_L,\mathsf{d}'_M,\mathsf{d}'_R,a'_D)$.

\begin{defi}\label{def:DUset}
Given $U = (\mathsf{u}_M,\mathsf{u}_R,\mathsf{u}'_L,\mathsf{u}'_M,\mathsf{u}'_R,\ell) \in \cW^5\times \N$, we define $\cD_U$ as the set of simple $\cW$-d.i. $D$ such that:
\begin{enumerate}
    \item \label{eq:def_S_xyzu:1}
    either $\mathsf{d}_M \in \cW^*\mathsf{u}_M$ or $\mathsf{d}_M = 1$ and $\mathsf{d}_L \leq_s \mathsf{u}_M$;
    \item \label{eq:def_S_xyzu:2}
    $\mathsf{d}_R a_D \leq_p \mathsf{u}_R$ and $|\mathsf{u_R}| = \min\{|w| : \mathsf{d}_Ra_D \leq_p w,\ w\in\cW\}$;
    \item \label{eq:def_S_xyzu:3}
    $\mathsf{d}'_Ra'_D\leq_p\mathsf{u}'_R$, $\mathsf{d}'_L\leq_s\mathsf{u}'_L$ and $|\mathsf{u}'_L| = \min\{|w| : \mathsf{d}'_L \leq_s w,\ w\in\cW\}$;
    \item \label{eq:def_S_xyzu:4}
    $\mathsf{d}'_M = 1$ or $\mathsf{d}'_M = v_1\cdots v_n \in \cW^+$, $v_1 = \mathsf{u}'_M$ and $\max_{1\leq j\leq n}|v_j|=\ell$.
\end{enumerate}
\end{defi}

It is easy to see that
\begin{equation*}\label{eq:def_cD}
\cD \coloneqq \bigcup_{U\in\cW^5\times\N} \cD_U
\end{equation*}
is the set of all simple $\cW$-d.i. of words in $\cA^+$. Moreover, from \eqref{eq:def_S_xyzu:4} of Definition \ref{def:DUset}
we have that $\ell \in \{|w| : w\in\cW\}\cup\{0\}$ when $\cD_U \not= \emptyset$, so $\cD$ is the union of no more than $\#\cW^5(\#\cW+1)$ sets $\cD_U$.

\begin{figure}[H]
\makebox[\textwidth][c]{\includegraphics[scale=1]{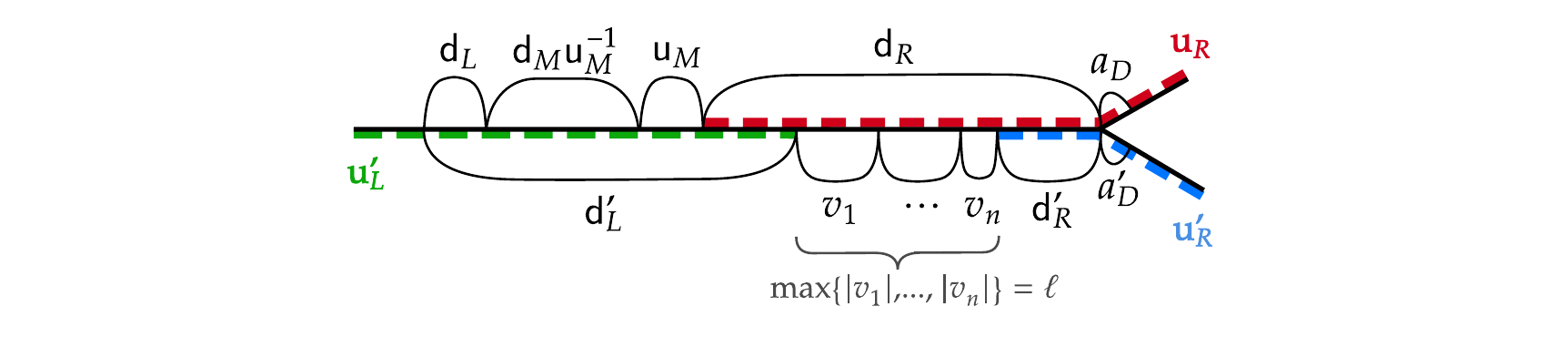}}
\caption{\label{fig:defi_U} Diagram illustrating  restrictions in Definition \ref{def:DUset} for a simple d.i. in the case $\mathsf{d}_M, \mathsf{d}'_M \not= 1$.}
\end{figure}

\begin{defi}\label{def:reduccion}
Let $D,E$ be simple d.i. on $\cW$. We say that,
\begin{enumerate}
\item $D$ is equivalent to $E$, and we write $D\sim E$, if $d$ and $e$ have a common suffix of length at least $\langle \cW\rangle$ (this makes sense by Remark \ref{rem:di_basic_props}).
\item $D$ reduces to $E$, and we write $D\Rightarrow E$, if $e <_s d$.
\end{enumerate}
\end{defi}
Observe that, when $D$ and $E$ are simple d.i. on $\cW$ with $D \Rightarrow E$, then, by Remark \ref{rem:di_basic_props}, $D\sim E$.

\begin{defi}
A subset  $\cD' \subseteq \cD$ of simple d.i. is {\em reducible} if
\begin{enumerate}
\item there are two different and equivalent elements in $\cD'$, or
\item there exists $D\in\cD'$ that reduces to some simple d.i.
\end{enumerate}
If $\cD'$ is not reducible, we say that it is {\em irreducible}.
\end{defi}

The main combinatorial result about irreducible sets of simple d.i. is the following proposition, whose proof will be carried out in Section \ref{sec:main_proof}.

\begin{prop}\label{lem:bound_on_irreducible}
Let $U \in \cW^5\times \N$. Any irreducible subset of $\cD_U$ has at most $61(\#\cW)$ elements.
\end{prop}
The use of condition (2) of Definition \ref{defi:double_interpretation} appears during the proof of this proposition. This proof consists in directly showing that sets $\cD' \subseteq \cD_U$ with more than $61(\#\cW)$ elements are reducible. For this, one finds elements in $\cD'$ that are equivalent or can be reduced. In this process, one observes that eliminating condition (2) in the definition of simple d.i. has two opposite effects. On one hand, it should be easier to find a reduction of a given simple d.i., since more d.i. are simple; but on the other hand, without condition (2) being simple means less structure, so it is more difficult to actually find the desired reductions during the proof. Balancing this trade-off is the reason behind the technical condition (2). It is worth mentioning that this condition (2) is only used in the proof of Lemma \ref{claim:regularidad}.

\section{Proof of main results}
\label{sec:main_results}

In this section we prove our main results. As we commented in the introduction, the proof of Theorem \ref{theo:main} is based on two general steps: first we use a proposition from \cite{DDPM15} relating the number of asymptotic components with the ``size'' of the automorphism group and secondly we develop a complete combinatorial analysis of the asymptotic classes arising in an $\cS$-adic subshift of finite alphabet rank.

Let $(X,T)$ be a topological dynamical system. Two points $x,x' \in X$ are (negatively) {\em asymptotic} if $\lim_{n\to -\infty}\mathrm{dist}(T^nx,T^nx') = 0$. We define the relation $\sim$ in $X$ as follows: $x \sim x'$ whenever $x$ is asymptotic to $T^kx'$ for some $k\in\Z$. It is easy to see that $\sim$ is an  equivalence relation. An equivalence class for $\sim$ that is not the orbit of a single point is called an {\em asymptotic class}, and we write $\mathrm{Asym}(X,T)$ for the set of asymptotic classes of $(X,T)$. Observe that if $(X,T)$ is a subshift, then $x \sim x'$ if and only if $x_{(-\infty,k)} = x'_{(-\infty,\ell)}$ for some $k,\ell \in \Z$.

The following proposition, which is a direct consequence of Corollary 3.3 in \cite{DDPM15}, gives a relation between the number of asymptotic classes and the cardinality of $\mathrm{Aut}(X,T)/\langle T\rangle$ under conditions that any infinite minimal subshift satisfies.

\begin{prop}\label{prop:action_autos_asym}
Let $(X,T)$ be a topological dynamical system. Assume there exists a point $x_0 \in X$ with $\omega(x_0) \coloneqq \bigcap_{n\geq0}\overline{\{T^kx_0:k\geq n\}} = X$ that is asymptotic to a different point. Then, $\#\mathrm{Aut}(X,T)/\langle T\rangle \leq \#\mathrm{Asym}(X,T)!$.
\end{prop}

Now we prove our first combinatorial theorem.

\begin{ctheo}[\ref{prop:main_technical}]
Let $\cW \subseteq \cA^+$ be a set of nonempty words. Then, there exists $\cB \subseteq \cA^{\langle\cW\rangle}$ with
$\#\cB \leq 122(\# \cW)^7$ such that:
if $x,x' \in \cA^\mathbb{Z}$ are factorizable over $\cW$, $x_{(-\infty,0)} = x'_{(-\infty,0)}$ and $x_0\not= x'_0$, then $x_{[-\langle \cW\rangle,0)} \in \cB$.
\end{ctheo}
As will be clear from the proof, the bound ``$122(\# \cW)^7$'' is not necessarily optimal. Here, the important point is that, despite the fact that the length of the elements in $\cB$ is $\langle\cW\rangle$, the cardinality of $\cB$ depends only on $\# \cW$, and not on $\langle \cW\rangle$. 
\begin{proof}
We start by defining the set $\cB$. For each $U = (\mathsf{u}_M,\mathsf{u}_R,\mathsf{u}'_L,\mathsf{u}'_M,\mathsf{u}'_R,\ell) \in \cW^5\times \N$, fix $\cD'_U \subseteq \cD_U$ an irreducible subset of maximal size (we consider the empty set as an irreducible set, so there always exists such set $\cD'_U$). We define
\begin{equation*}
\cB \coloneqq 
\big\{w \in\cA^{\langle\cW\rangle}: 
\exists\ U \in \cW^5\times \N,\ D \in \cD'_U,\ w\leq_s d \big\},
\end{equation*}
where in this set $d\in\cA^+$ represents the word that is double interpreted by $D$.
We note that this makes sense because $|d|\geq\langle \cW\rangle$ for all simple d.i. As we observed previously, we have $\ell \in \{|w|:w\in\cW\}\cup\{0\}$ when $\cD_U$ is nonempty. Thus, there are no more than $\#\cW^5(\#\cW+1)$ choices for $U$ such that $\cD_U$ is nonempty. Using this and Proposition \ref{lem:bound_on_irreducible} we get:
\begin{align*}
\#\cB &\leq 61\# \cW\cdot \#\{U\in \cW^5\times\N:\cD_U\not=\emptyset\}, \\
&\leq 61\# \cW\cdot \# \cW^5(\# \cW + 1) \leq 122(\#\cW)^7.
\end{align*}
It left to prove the main property of the theorem. In this purpose, 
let $x,x' \in \cA^\Z$ be factorizable over $\cW$ with $x_{(-\infty,0)} = x'_{(-\infty,0)}$ and $x_0\not= x'_0$.
From Lemma \ref{lem:asymptotic_to_di} we can find a simple d.i. $D$ of $d \leq_s x_{(-\infty,0)}$. Let
\begin{equation*}
D \eqqcolon D(0) \Rightarrow D(1) \Rightarrow D(2) \Rightarrow \dots \Rightarrow D(n)
\end{equation*}
be a sequence of reductions that starts with $D$ (where, possibly, $n=0$ and $D$ has no reduction). We write, for convenience, $D(j) = (I(j);I'(j))$ and $d(j)$ for the word that is double interpreted by $D(j)$. Since $|d(0)| > |d(1)| > \dots$, any sequence like this ends after a finite number of steps. In particular, we can take (and we \textit{are} taking) this sequence so that $n$ is maximal. This implies that $D(n)$ has no reduction.

Since $\cD = \bigcup_{U \in \cW^5\times \N} \cD_U$, we can find $U \in \cW^5\times \N$ satisfying $D(n) \in \cD_U$. We claim that there is a word $e$ with a simple d.i. $E = (I_E;I'_E)\in\cD'_U$ such that $D(n)$ is equivalent to $E$. Indeed, if $D(n)\in\cD'_U$ then, since $D(n)$ is equivalent to itself, we can take $E\coloneqq D(n)$. If $D(n)$ is not in $\cD'_U$, then, from the maximality of $\cD'_U$ we see that $\cD'_U \cup \{D(n)\}$ is reducible. Since $D(n)$ has no reduction and $\cD'_U$ is irreducible, there exists $E \in \cD'_U$ equivalent to $D(n)$. This proves the claim. 

Then, using the definitions of reduction and equivalence of simple d.i., we have that the suffix $w \in \cA^{\langle\cW\rangle}$ of $e$ satisfies
\begin{align*}
w \leq_s d(n) <_s
d(n-1) <_s \dots <_s d(0) \leq_s x_{(-\infty,0)},
\end{align*}
and $w \in \cB$ since $E \in \cD'_U$. This finishes the proof.
\end{proof}

Now we have all the ingredients to compute the number of asymptotic classes in the case of $\cS$-adic subshifts of finite alphabet rank.

\begin{ctheo}[\ref{theo:s_adic_finite_asymptotic}]
Let $(X,T)$ be an $\cS$-adic subshift given by an everywhere growing directive sequence of alphabet rank $K$. Then, $(X,T)$ has at most $122K^7$ asymptotic classes.
\end{ctheo}

\begin{proof}
Set $K' =  122K^7$. We are going to prove the following stronger result.
\begin{claim}
Let $\mathcal{P}$ be the set of pairs $(x,y) \in X\times X$ such that $x_{(-\infty, 0)} = y_{(-\infty,0)}$ and $x_0 \not= y_0$. Then, $\#\{x_{(-\infty,0)} : (x,y) \in \mathcal{P}\} \leq K'$.
\end{claim}
First, we show how this claim implies the theorem. Suppose the claim is true and let $C_0,\ldots, C_{K'}$ be asymptotic classes for $(X,T)$. For each $j\in \{0,\ldots,K'\}$ we choose $(z_j,z'_j) \in C_j$ such that $z_j$ and $z'_j$ do not belong to the same orbit. Then, there exist $m_j,m'_j \in \Z$ such that $x_j \coloneqq T^{m_j}z_j$ and $y_j \coloneqq T^{m'_j}z'_j$ satisfy
\begin{equation}\label{eq:theo:s_adic_finite_asymptotic:0}
    (x_j)_{(-\infty,0)} = (y_j)_{(-\infty,0)}
    \quad\text{and}\quad (x_j)_0 \not= (y_j)_0,\quad \forall j\in \{0,\ldots,K'\}.
\end{equation}
Thus, $(x_j,y_j) \in \mathcal{P}$ for all $j \in \{0,\ldots,K'\}$ and, by the claim and the Pigeonhole Principle, there exist different $j,j' \in \{0,\ldots,K'\}$ such that $(x_j)_{(-\infty,0)} = (x_{j'})_{(-\infty,0)}$. This implies $C_j = C_{j'}$ and, thus, that $(X,T)$ has at most $K'$ asymptotic classes.

Now we prove the claim. Let $\boldsymbol{\tau} = (\tau_n\colon \cA_{n+1}^+\to \cA_n^+)_{n\geq0}$ be an everywhere growing directive sequence of alphabet rank $K$ generating $X$. By doing a contraction, if required, we can suppose that $\# \cA_n = K$ for every $n\geq 1$. For $n\geq 1$ put $\cW_n = \tau_{[0,n)}(\cA_n)$ and let $\cB_n \subseteq \cA_0^+$ be the set given by Theorem \ref{prop:main_technical} when it is applied to $\cW_n$. By hypothesis, $\# \cW_n \leq \#\cA_n = K$, so $\# \cB_n \leq 122(\#\cW_n)^7 \leq 122K^7=K'$.

For $j\in \{0,\ldots,K'\}$ let $(x_j,y_j) \in \mathcal{P}$. We have to show that $(x_j)_{(-\infty,0)} = (x_{j'})_{(-\infty,0)}$ for different $j,j' \in \{0,\ldots,K'\}$. Since for all $n\geq 1$ and $j\in\{0,\ldots,K'\}$ the points $x_j$ and $y_j$ are factorizable over $\cW_n$ (Lemma \ref{lem:desubstitution}), from Theorem \ref{prop:main_technical} we have that $(x_{j})_{[-\langle  \cW_n\rangle , 0)} \in \cB_n$. But $\#\cB_n\leq K'$ so by the Pigeonhole Principle there exist $j_n,j_n' \in \{0,\ldots,K'\}$ with $j_n\not=j'_n$ such that
\begin{equation}\label{eq:theo:s_adic_finite_asymptotic:1}
(x_{j_n})_{[-\langle  \cW_n\rangle , 0)} = (x_{j'_n})_{[-\langle  \cW_n\rangle , 0)}.
\end{equation}
Thus, again by the Pigeonhole Principle,  
we can choose $1\leq n_1 < n_2 < \dots$ such that $j_{n_1} = j_{n_2} = \dots = j \not= j' = j'_{n_1} = j'_{n_2} = \dots$ By (\ref{eq:theo:s_adic_finite_asymptotic:1}),
\begin{equation}\label{eq:theo:s_adic_finite_asymptotic:2}
(x_j)_{[-\langle  \cW_{n_i}\rangle , 0)} = (x_{j'})_{[-\langle  \cW_{n_i}\rangle , 0)},
\qquad \forall i\geq 1.
\end{equation}
Since $\boldsymbol{\tau}$ is everywhere growing, $\langle  \cW_n\rangle$ goes to infinity when $n\to +\infty$. Thus, (\ref{eq:theo:s_adic_finite_asymptotic:2}) implies that $(x_j)_{(-\infty,0)} = (x_{j'})_{(-\infty,0)}$, as desired. This completes the proof.
\end{proof}

We remark again that in previous result we do not assume minimality. This hypothesis is needed in the next proof (of Theorem \ref{theo:main}) only because we bound the size of the automorphism group by the number of asymptotic classes via Proposition \ref{prop:action_autos_asym}. Thus, Theorem \ref{theo:main} is mainly a consequence of combinatorial facts inherent to $\cS$-adic subshifts.

\begin{ctheo}[\ref{theo:main}]
Let $(X,T)$ be a minimal $\cS$-adic subshift given by an everywhere growing sequence of finite alphabet rank $K$. Then, its automorphism group is virtually $\Z$.
\end{ctheo}
\begin{proof}
From Proposition \ref{prop:action_autos_asym} and Theorem \ref{theo:s_adic_finite_asymptotic} we get
\begin{equation*}
\#\mathrm{Aut}(X,T)/\langle T\rangle \leq 
\#\mathrm{Asym}(X,T)! \leq 
\left(122K^7\right)! < +\infty.
\end{equation*}
This inequality proves that $\mathrm{Aut}(X,T)$ is virtually $\Z$.
\end{proof}

\section{Proof of Proposition \ref{lem:bound_on_irreducible}}
\label{sec:main_proof}

In this last section we prove Proposition \ref{lem:bound_on_irreducible}. All but one result  we need (Lemma \ref{lem:extract_di}) are presented and proved here, so the section is almost self contained. 

We fix, for the rest of this section, a finite set of words $\cW \subseteq \cA^+$ and a sequence $U = (\mathsf{u}_M,\mathsf{u}_R,\mathsf{u}'_L,\mathsf{u}'_M,\mathsf{u}'_R,\ell) \in \cW^5\times \N$. For $D \in \cD_U$, we define:
\begin{align*}
\mathsf{\tilde{d}}&\coloneqq 
\mathsf{d}_R (\mathsf{d}'_M\mathsf{d}'_R)^{-1} = 
(\mathsf{d}_L\mathsf{d}_M)^{-1} \mathsf{d}'_L.
\end{align*}


We need a last definition: two words $u,v \in \cA^*$ are \textit{prefix dependent} (resp. suffix dependent) if $u \leq_p v$ or $v \leq_p u$ (resp. $u \leq_s v$ or $v \leq_s u$). In this case, $u$ and $v$ share a common prefix (resp. suffix) of length $\min(|u|,|v|)$.

\begin{lem}\label{lem:redux_if}
Consider different elements $D,E$ in $\cD_U$. If any of the following conditions holds, then the set $\{D,E\}$ is reducible:
\begin{enumerate}[label=(\roman*)]
\item \label{lem:redux_if:prefix}
$\mathsf{d}'_M \mathsf{d}'_R a_D$, $\mathsf{e}'_M \mathsf{e}'_R a_E$ are prefix dependent;

\item\label{lem:redux_if:alpha_prime}
$|\mathsf{d}_R| = |\mathsf{e}_R|$;

\item\label{lem:redux_if:alpha_bar}
$|\mathsf{\tilde{d}}| \leq |\mathsf{\tilde{e}}| \leq |\mathsf{\tilde{d}}\mathsf{d}'_M|$ or $|\mathsf{\tilde{e}}| \leq |\mathsf{\tilde{d}}| \leq |\mathsf{\tilde{e}}\mathsf{e}'_M|$.
\end{enumerate}
\end{lem}
\begin{proof}
We will show that under conditions of the lemma one of the following relations occurs: $D\sim E$, $E$ reduces to a simple d.i. or $D$ reduces to a simple d.i.
\begin{enumerate}[leftmargin=0cm, label=\textbf{(\roman*)}]
\item Without loss of generality, we can suppose that $\mathsf{d}'_M \mathsf{d}'_R a_D \leq_p \mathsf{e}'_M \mathsf{e}'_R a_E$. We distinguish two cases:
\begin{enumerate}[wide, labelwidth=!, labelindent=0pt]

\item $\mathsf{d}'_M \mathsf{d}'_R a_D = \mathsf{e}'_M \mathsf{e}'_R a_E$. Using item \eqref{eq:def_S_xyzu:3} of Definition \ref{def:DUset} we can write $d = \mathsf{d}'_L\mathsf{d}'_M\mathsf{d}'_R \leq_s \mathsf{u}'_L\mathsf{d}'_M\mathsf{d}'_R$. Similarly, $e \leq_s \mathsf{u}'_L\mathsf{e}'_M\mathsf{e}'_R$. This and hypothesis (a) imply that $d$ and $e$ are suffix dependent. But, since $D$ and $E$ are simple d.i., by Remark \ref{rem:di_basic_props} we have that $|d|,|e| \geq \langle\cW\rangle$. We conclude that $d$ and $e$ share a suffix of length at least $\min(|d|,|e|) \geq \langle \cW\rangle$, which implies $D\sim E$.

\item $\mathsf{d}'_M \mathsf{d}'_R a_D <_p \mathsf{e}'_M \mathsf{e}'_R a_E$ (so, $\mathsf{d}'_M \mathsf{d}'_Ra_D \leq_p \mathsf{e}'_M \mathsf{e}'_R$).  We claim that $\ell > 0$ in the definition of $U$. Suppose that $\ell = 0$. Then, $\mathsf{d}'_M = \mathsf{e}'_M = 1$ and we can write:
\begin{align*}
\mathsf{d}'_R a_D \leq_p \mathsf{e}'_R \leq_p \mathsf{u}'_R.
\end{align*}
Since by \eqref{eq:def_S_xyzu:3} of Definition \ref{def:DUset} we also have $\mathsf{d}'_R a'_D \leq_p \mathsf{u}'_R$, we conclude that $a_D = a'_D$. This contradicts the fact that $E$ is a d.i. Thus, $\ell > 0$.

Now, $\ell > 0$ and \eqref{eq:def_S_xyzu:4} of Definition \ref{def:DUset} imply that $v_D \coloneqq (\mathsf{u}'_M)^{-1}\mathsf{d}'_M \in \cW^*$ and $v_E \coloneqq (\mathsf{u}'_M)^{-1}\mathsf{e}'_M \in \cW^*$. Let $w \coloneqq \mathsf{d}'_M \mathsf{d}'_R$. Observe that $J_D = \mathsf{u}'_M, v_D, \mathsf{d}'_R, a'_D$ is an interpretation of $w$. Moreover, since $\mathsf{u}'_M \leq_p w <_p \mathsf{u}'_M v_E \mathsf{e}'_R$ by hypothesis (b) and $v_E \in \cW^*$, we can obtain, using Lemma \ref{lem:inherit}, an interpretation of $w$ of the form $J_E = \mathsf{u}'_M, \mathsf{e}''_M, \mathsf{e}''_R, a''_E$ such that $w a''_E \leq_p \mathsf{u}'_M v_E\mathsf{e}'_R$.

Next, we prove that $F \coloneqq (J_D; J_E)$ is a d.i. of $w$. Observe that  $v_D\mathsf{d}'_Ra_D \leq_p v_E\mathsf{e}'_R$ by hypothesis (b) and $\mathsf{e}''_M\mathsf{e}''_Ra''_E \leq_p v_E\mathsf{e}'_R$ by the definition of $J_E$. But $v_D\mathsf{d}'_R = (\mathsf{u}'_R)^{-1}w = \mathsf{e}''_M\mathsf{e}''_R$, so $a_D = a''_E$. Hence, $a'_D \not= a_D = a''_E$ and $F$ is a d.i. of $w$.

Finally, we note that since $J_D$ and $J_E$ start with $\mathsf{u}'_M \in \cW$, we can use Lemma \ref{lem:extract_di} with $F$ to obtain a simple d.i. $G$ of a word $g$ such that $g\leq_s w <_s d$. This corresponds to the fact that $D$ reduces to $G$.

\end{enumerate}

\item Assume $|\mathsf{d}_R| = |\mathsf{e}_R|$. Since, by \eqref{eq:def_S_xyzu:2} of Definition \ref{def:DUset}, we have that $\mathsf{d}_R$ and $\mathsf{e}_R$ are prefix of $\mathsf{u}_R$, hypothesis (ii) implies that $\mathsf{d}_R = \mathsf{e}_R$. In addition, from \eqref{eq:def_S_xyzu:1} of Definition \ref{def:DUset} we see that $\mathsf{d}_L\mathsf{d}_M$ and $\mathsf{e}_L\mathsf{e}_M$ either share the suffix $\mathsf{u}_M\in \cW$ or are suffix dependent. We conclude that $d = \mathsf{d}_L\mathsf{d}_M\mathsf{d}_R$ and $e = \mathsf{e}_L\mathsf{e}_M\mathsf{e}_R$ share a suffix of length at least $\langle\cW\rangle$. This is, $D \sim E$.

\item We consider the case $|\mathsf{\tilde{d}}| \leq |\mathsf{\tilde{e}}| \leq |\mathsf{\tilde{d}}\mathsf{d}'_M|$, the other one is symmetric. 

We start with some simplifications. Observe that condition \eqref{eq:def_S_xyzu:2} in Definition \ref{def:DUset} implies
\begin{equation}\label{eq:lem:redux_if:1.5}
\mathsf{d}_Ra_D =
\mathsf{\tilde{d}}\mathsf{d}'_M \mathsf{d}'_R a_D \leq_p \mathsf{u}_R
\quad\text{and}\quad
\mathsf{e}_R a_E =
\mathsf{\tilde{e}}\mathsf{e}'_M \mathsf{e}'_R a_E \leq_p \mathsf{u}_R.
\end{equation}
Then, if $|\mathsf{\tilde{d}}| = |\mathsf{\tilde{e}}|$, we are in case \ref{lem:redux_if:prefix}, and if $|\mathsf{d}_R| = |\mathsf{e}_R|$, we are in case \ref{lem:redux_if:alpha_prime}. Thus, we can suppose, without loss of generality, that
\begin{align}
|\mathsf{\tilde{d}}| &< |\mathsf{\tilde{e}}|, \label{eq:lem:redux_if:2}\\
|\mathsf{d}_R| &\not= |\mathsf{e}_R|. \label{eq:lem:redux_if:3}
\end{align}

The idea of the proof is the following. We are going to define a word $w$, which is suffix of $d$ or $e$, and that has a d.i. $F$ satisfying the hypothesis of Lemma \ref{lem:extract_di}. This would imply that $F$ (and then also $D$ or $E$) reduces to a simple d.i., as desired.

From \eqref{eq:lem:redux_if:2} and hypothesis (iii) we have that $|\mathsf{\tilde{d}}| \not = |\mathsf{\tilde{d}}\mathsf{d}'_M|$ and thus $\ell\not =0$. In particular, this last fact implies that $v_D \coloneqq (\mathsf{u}'_M)^{-1}\mathsf{d}'_M \in \cW^*$ and  $v_E \coloneqq (\mathsf{u}'_M)^{-1}\mathsf{e}'_M \in \cW^*$. Also, from \eqref{eq:lem:redux_if:1.5} and \eqref{eq:lem:redux_if:2} we see that it makes sense to define $t \coloneqq \mathsf{\tilde{d}}^{-1}\mathsf{\tilde{e}} \not= 1$. Then, $J_D = \mathsf{u}'_M, v_D, \mathsf{d}'_R, a'_D$ is an interpretation of $\mathsf{d}'_M \mathsf{d}'_R$ and $J_E = t, \mathsf{e}'_M, \mathsf{e}'_R, a'_E$ is an interpretation of $t \mathsf{e}'_M \mathsf{e}'_R$. Now, using \eqref{eq:lem:redux_if:1.5} and \eqref{eq:lem:redux_if:3} we also obtain that either $\mathsf{d}'_M \mathsf{d}'_R <_p t \mathsf{e}'_M \mathsf{e}'_R$ or $t\mathsf{e}'_M \mathsf{e}'_R <_p \mathsf{d}'_M \mathsf{d}'_R$. We analyze these two cases separately:

\medskip
\begin{enumerate}[wide, labelwidth=!, labelindent=0pt]
\item Assume $\mathsf{d}'_M \mathsf{d}'_R <_p t \mathsf{e}'_M \mathsf{e}'_R$. We define $w = \mathsf{d}'_M \mathsf{d}'_R <_s d$. Note that $J_D$ is an interpretation of $w$. By hypothesis (iii), we have $t \leq_p w <_p t \mathsf{e}'_M\mathsf{e}'_R$, so we can use Lemma \ref{lem:inherit} with $J_E$ to obtain an interpretation of $w$ having the form $J'_E = t, \mathsf{e}''_M, \mathsf{e}''_R, a$ and satisfying $wa \leq_p \mathsf{e}'_M\mathsf{e}'_R$. We set $F = (J_D, J'_E)$. Since $wa \leq_p t\mathsf{e}'_M\mathsf{e}'_R = \mathsf{\tilde{d}}^{-1}\mathsf{e}_R \leq_p \mathsf{\tilde{d}}^{-1}\mathsf{u}_R$ and $wa_D = \mathsf{d}'_M\mathsf{d}'_Ra_D = \mathsf{\tilde{d}}^{-1}\mathsf{d}_Ra_D \leq_p \mathsf{\tilde{d}}^{-1}\mathsf{u}_R$, we have $a = a_D$. Being $a_D \not= a'_D$ as $D$ is a d.i., we conclude that $a \not= a'_D$ and that $F$ is a d.i. Recall that $\mathsf{u}'_R \in \cW$ and observe that $|t| \leq |\mathsf{d}'_M|$ by hypothesis (iii). Thus, $F$ satisfies the hypothesis of Lemma \ref{lem:extract_di}. This implies that $D$ is reducible.

\medskip
\item Suppose $t \mathsf{e}'_M \mathsf{e}'_R <_p \mathsf{d}'_M \mathsf{d}'_R$. Observe that from \eqref{eq:def_S_xyzu:4} of Definition \ref{def:DUset} we know that there exist $n\geq0$ and, for $j\in\{1,\dots,n\}$, $v_j \in \cW$ with $|v_j|\leq\ell$, such that $v_D = v_1\cdots v_n$ (we interpret $v_1\cdots v_n = 1$ when $n=0$). We define $w = t \mathsf{e}'_M \mathsf{e}'_R <_s e$ and $v_{n+1} = \mathsf{d}'_R$. Since $|w| \geq |\mathsf{u}'_R|$, we have $\mathsf{u}'_M \leq_p w <_p \mathsf{u}'_M v_1\cdots v_{n+1}$ by (b), and thus, there exists a least integer $m \in \{1,\dots,n+1\}$ such that $w \leq_p \mathsf{u}'_M v_1\cdots v_m$. Being $m$ minimal, we can write $w = \mathsf{u}'_M v_1\cdots v_{m-1} v'_m$, with $v'_m \leq_p v_m$ and $wa \leq_p \mathsf{d}'_M\mathsf{d}'_R$ for some $a\in\cA$. Then, $J'_D \coloneqq \mathsf{u}'_M, v_1\cdots v_{m-1}, v'_m, a$ and $J_E$ are interpretations of $w$. 

We set $F = (J'_D, J_E)$ and claim that $F$ is a d.i. Indeed, on the one hand, the definition of $J'_D$ gives $wa \leq_p \mathsf{d}'_M\mathsf{d}'_R \leq_p \mathsf{\tilde{d}}^{-1}\mathsf{u}_R$. On the other hand, since $w = \mathsf{\tilde{d}}^{-1}\mathsf{\tilde{e}}\mathsf{e}'_M\mathsf{e}'_R = \mathsf{\tilde{d}}^{-1}\mathsf{e}_R$, we have $wa_E \leq_p \mathsf{\tilde{d}}^{-1}\mathsf{u}_R$ by \eqref{eq:def_S_xyzu:2} of Definition \ref{def:DUset}. We conclude that $a = a_E$. Then, $a \not= a'_E$ (because $E$ is a d.i.) and $F$ is a d.i.

Finally, we prove that $F$ satisfies the hypothesis of Lemma \ref{lem:extract_di}. Since $J'_D$ starts with $\mathsf{u}'_M \in \cW$, we only need to show that $|t| \leq |\mathsf{u}'_Mv_1\cdots v_{m-1}|$. By contradiction, we assume $\mathsf{u}'_Mv_1\cdots v_{m-1} <_p t$. This condition implies two things. First, that we can define $t' = (\mathsf{u}'_Mv_1\cdots v_{m-1})^{-1}t \not=1$, and then, since $\mathsf{u}'_Mv_1\cdots v_{m-1}v'_m = t\mathsf{e}'_M\mathsf{e}'_R$, that $v'_m = t'\mathsf{e}'_M\mathsf{e}'_R$. In particular, $\ell \leq |\mathsf{e}'_M| < |v'_m|$. The second fact is that $m \leq n$. Indeed, by hypothesis (iii) we have $|\mathsf{u}'_Mv_1\cdots v_{m-1}| < |t| \leq |\mathsf{d}'_M| = |\mathsf{u}'_M v_1\cdots v_n|$. Hence, $\ell < |v'_m| \leq |v_m| \leq \ell$, which is a contradiction. This proves that Lemma \ref{lem:extract_di} can be applied with $F$, so $F$ (and then also $E$) reduces to a simple d.i.
\end{enumerate}

\end{enumerate}
\end{proof}

If $u \in \cA^+$, then we write $u^\infty \coloneqq uuu\cdots$ and $\prescript{\infty}{}{u} \coloneqq \cdots uuu$. Recall that an  integer $k\geq1$ is a period of $w \in \cA^+$ if $w \leq_p u^\infty$ (equivalently, $w \leq_s \prescript{\infty}{}{u}$) for some $u \in \cA^k$. The following result (also known as the Fine and Wilf Lemma) is classical.
\begin{lem}[Proposition 1.3.2, \cite{Lothaire}]\label{lem:fine_wilf}
If $p,p' \geq 1$ are periods of $w \in \cA^+$ and $p+p' \leq |w|$, then $\gcd(p,p')$ is also a period of $w$.
\end{lem}

We fix an irreducible subset $\cD' \subseteq \cD_U$. For $D,E \in \cD'$, since $\mathsf{\tilde{d}}, \mathsf{\tilde{e}} \leq_p \mathsf{u}_R$ and $\mathsf{\tilde{d}}, \mathsf{\tilde{e}} \leq_s \mathsf{u}'_L$, we have that $\mathsf{\tilde{d}}$ and $\mathsf{\tilde{e}}$ are  
both prefix and suffix dependent. So it makes sense to define in $\cD'$: 
\begin{equation*}
D \leq E
\quad \text{iff} \quad
\mathsf{\tilde{d}} \leq_p \mathsf{\tilde{e}}.
\end{equation*}
Observe that Lemma \ref{lem:redux_if} part \ref{lem:redux_if:alpha_bar} implies that $D = E$ if and only if $\mathsf{\tilde{d}} =\mathsf{\tilde{e}}$. Therefore, $\leq$ is a total order. In particular, we can use the notation $D < E$ when $D\leq E$ and $D\not= E$. In this case it is not difficult to prove that $|\mathsf{\tilde{e}}|-|\mathsf{\tilde{d}}|$ is a period of $\mathsf{\tilde{e}}$.

Let $D(1) < \dots < D(s)$ be all the elements in $\cD'$ (deployed in increasing order). We adopt the mnemotechnical notation: 
\begin{align}\label{eq:mnemotechnical}
&D(j) = (\mathsf{d}_L(j), \mathsf{d}_M(j), \mathsf{d}_R(j), a(j);\ \mathsf{d}'_L(j), \mathsf{d}'_M(j), \mathsf{d}'_R(j), a'(j)); \\
& d(j) = \mathsf{d}_L(j) \mathsf{d}_M(j) \mathsf{d}_R(j), \
\mathsf{\tilde{d}}(j) = (\mathsf{d}_L(j) \mathsf{d}_M(j))^{-1}\mathsf{d}'_L(j).
\end{align}

For $D, E \in \cD'$, since $\mathsf{d}_Ra_D, \mathsf{\tilde{e}} \leq_p \mathsf{u}_R$, we have that $\mathsf{d}_Ra_D \leq_p \mathsf{\tilde{e}}$ if and only if $|\mathsf{d}_R| < |\mathsf{\tilde{e}}|$. Thus, for $j \in \{1,\dots,s\}$ we can define
\begin{equation*}
\cD' (j) \coloneqq
\{D \in \cD' : \mathsf{d}_Ra_D \leq_p \mathsf{\tilde{d}}(j)\} =
\{D \in \cD' : |\mathsf{d}_R| < |\mathsf{\tilde{d}}(j)|\}
\end{equation*}
and $\cD' (s+1) \coloneqq \cD' $. By definition of the total order, this is a nondecreasing sequence. Moreover,  $\cD'(j) \subseteq \{D(k):k\in \{1,\dots,j-1\}\}$ for all $j \in \{1,\dots,s+1\}$. In particular,  $\cD'(1) = \emptyset$.

\begin{lem}\label{claim:cota_sincronos}
Let $p \in \{1,\dots,s+1\}$ be such that $\cD'(p)$ is nonempty and let $D(p') \coloneqq \max\cD'(p)$, where the maximum is taken with respect to the total order. Then, $\#(\cD'(p)\backslash \cD'(p')) \leq 6$.
\end{lem}
\begin{proof}
We prove the lemma by contradiction. Suppose $\#(\cD' (p)\backslash \cD' (p')) \geq 7$ and let $D(j_1) < D(j_2) < \dots < D(j_7)$ be seven different elements in $\cD'(p)\backslash \cD'(p')$. 

We start by obtaining some relations. First, from part \ref{lem:redux_if:alpha_bar} of Lemma \ref{lem:redux_if} and the irreducibility of $\cD'$, we get
\begin{equation}\label{eq:claim:cota_sincronos:1}
\text{$\mathsf{\tilde{d}} \mathsf{d}'_M <_p \mathsf{\tilde{e}}$ for all $D,E \in \cD' (p)$ such that $D < E$}.
\end{equation}
Thus,
\begin{equation}\label{eq:claim:cota_sincronos:2}
\mathsf{\tilde{d}}(j_k) \leq_p \mathsf{\tilde{d}}(j_k) \mathsf{d}'_M(j_k) <_p
\mathsf{\tilde{d}}(j_{k+1}) \leq_p \mathsf{\tilde{d}}(j_{k+1})\mathsf{d}'_M(j_{k+1})
\text{ for all } k\in\{1,\dots,6\}.
\end{equation}
We set $v_k = \mathsf{\tilde{d}}(j_k)\mathsf{d}'_M(j_k)$, $k \in \{1,\dots,6\}$. By \eqref{eq:claim:cota_sincronos:2},
\begin{equation*}
v_1 <_p \dots <_p v_5 <_p \mathsf{\tilde{d}}(j_6) <_p v_6 <_p \mathsf{\tilde{d}}(j_7).
\end{equation*}
Also, observe that for any $D \in \cD'(p)\backslash \cD'(p')$ we have $D \leq D(p')$ and $D\not\in\cD'(p')$, which gives
\begin{equation}\label{eq:claim:cota_sincronos:3}
\mathsf{\tilde{d}} \leq_p \mathsf{\tilde{d}}(p') \leq_p \mathsf{d}_R \leq_p \mathsf{u}_R.
\end{equation}

Equation \eqref{eq:claim:cota_sincronos:2}, the first inequality of \eqref{eq:claim:cota_sincronos:3} used with $\mathsf{d}(j_7)$ and the second inequality of \eqref{eq:claim:cota_sincronos:3} used with $\mathsf{\tilde{d}}(j_k)$ imply that
\begin{equation}\label{eq:claim:cota_sincronos:4}
v_k <_p \mathsf{\tilde{d}}(j_7) \leq_p 
\mathsf{\tilde{d}}(p') \leq_p
\mathsf{d}_R(j_k) \text{ for all } k\in\{1,\dots,6\}. 
\end{equation}

\medskip
From previous relations we can define the nonempty word $w \coloneqq v_1^{-1} \mathsf{\tilde{d}}(j_7)$. Let $q\leq_p w$ be such that $|q|$ is the least period of $w$. We will prove that $|q|$ divides $|v_1^{-1}v_k|$ for all $k \in \{1,\dots,5\}$. 

On the one hand, the observation made before the proof shows that $|\mathsf{\tilde{d}}(j_6)^{-1}\mathsf{\tilde{d}}(j_7)|$ is a period of $\mathsf{\tilde{d}}(j_7)$, and thus also of $w$. On the other hand, if $k \in \{1,\dots,6\}$, then from \eqref{eq:claim:cota_sincronos:4} and the definition of $\mathsf{\tilde{d}}$ we get 
\begin{equation*}
(v_1^{-1}v_k)^{-1}w = v_k^{-1}\mathsf{\tilde{d}}(j_7) \leq_p v_k^{-1}\mathsf{d}_R(j_k) = \mathsf{d}'_R(j_k) \leq_p \mathsf{u}'_R,
\end{equation*}
being the last step true due to item \eqref{eq:def_S_xyzu:3} of Definition \ref{def:DUset}. In particular, for $k=1$ we get $w \leq_p \mathsf{u}'_R$. These inequalities imply $w \leq_p (v_1^{-1}v_k)^\infty$. Consequently, $|v_1^{-1}v_k|$ is a period of $w$. Since, by \eqref{eq:claim:cota_sincronos:2}, $v_k^{-1}\mathsf{\tilde{d}}(j_6)$ is defined for all $k \in \{1,\dots,5\}$, then for these values of $k$ we can compute 
\begin{equation*}
|q| + |v_1^{-1}v_k| \leq
|\mathsf{\tilde{d}}(j_6)^{-1}\mathsf{\tilde{d}}(j_7)| + |v_1^{-1}v_k| = |w| - |v_k^{-1}\mathsf{\tilde{d}}(j_6)| \leq |w|.
\end{equation*}
Hence, Lemma \ref{lem:fine_wilf} can be applied to get that $\gcd( |q|, |v_1^{-1}v_k|)$ is a period of $w$ for $k \in \{1,\dots,5\}$. In particular, $|q| = \gcd( |q|, |v_1^{-1}v_k|)$ and $|q|$ divides $|v_1^{-1}v_k|$ for $k \in \{1,\dots,5\}$.
\medskip

Then, we have $w \leq_p q^\infty$ and, by the claim, for $k \in \{1,\dots,5\}$ there exists $n_k \geq 0$ satisfying $v_1^{-1}v_k = q^{n_k}$.
Moreover, from the definition of $v_k$ we have $v_k = v_1q^{n_k}$, which implies
\begin{align*}\label{eq:claim:cota_sincronos:5}
\mathsf{d}'_R(j_k) a(j_k) = 
v_k^{-1} \mathsf{d}_R(j_k)a(j_k) \leq_p
v_k^{-1} \mathsf{u}_R =
q^{-n_k} v_1^{-1} \mathsf{u}_R
\end{align*}
and $\mathsf{d}'_R(j_k) a'(j_k) \leq_p \mathsf{u}'_R$. Thus, since $a(j_k) \not= a'(j_k)$, we deduce that $\mathsf{d}'_R(j_k)$ is the maximal common prefix of $q^{-n_k} v_1^{-1}\mathsf{u}_R$ and $\mathsf{u}'_R$. 

Now, let $n,n' \geq 0$ and $r,r' <_p q$ be maximal such that $q^n r \leq_p v_1^{-1}\mathsf{u}_R$ and $q^{n'} r' \leq_p \mathsf{u}'_R$. We conclude that
\begin{equation}\label{eq:claim:cota_sincronos:6}
\text{$\mathsf{d}'_R(j_k) = q^{n-n_k}r$ if $n-n_k < n'$ and $\mathsf{d}'_R(j_k) = q^{n'}r'$ if $n-n_k > n'$}
\end{equation}
for $k \in \{1,\dots,5\}$.

\medskip
We have all the elements to complete the proof. Since $n_2 < n_3 < n_4 < n_5$, we have $n_2 < n_3 < n-n'$ or $n_5 > n_4 > n-n'$. We are going to show that both cases give a contradiction, proving, thereby, the lemma. 

First, suppose that $n_2 < n_3 < n-n'$. Then, for $k \in \{2,3\}$, we have $n - n_k > n'$, and thus, by \eqref{eq:claim:cota_sincronos:6}, $\mathsf{d}'_R(j_k) = q^{n'}r'$. If $\ell = 0$, $d(j_k) = \mathsf{d}'_L(j_k) \mathsf{d}'_R(j_k) \leq_s \mathsf{u}'_L q^{n'}r'$. Then, $d(j_2)$ and $d(j_3)$ are suffix dependent, which gives that $D(j_2)$ is equivalent to $D(j_3)$, contradicting the irreducibility of $\cD'$. 
If $\ell > 0$, we have $\mathsf{d}_R(j_k) = v_1 (v_1^{-1}v_k) \mathsf{d}'_R(j_k) = v_1 q^{n_k+n'}r'$. Then, using \eqref{eq:claim:cota_sincronos:2},
\begin{equation*}
|q^{n_k}| = |v_1^{-1}v_k| \geq 
|v_1^{-1}v_2| \geq |\mathsf{d}'_M(j_2)| \geq |\mathsf{u}'_M| \geq \langle\cW\rangle,
\end{equation*}
and hence $d(j_2)$ and $d(j_3)$ share a common suffix of length $\langle\cW\rangle$. This is, $D(j_2) \sim D(j_3)$, which is a contradiction.

Finally, assume $n_5 > n_4 > n-n'$. We have, by \eqref{eq:claim:cota_sincronos:6}, that $\mathsf{d}'_R(j_k) = q^{n-n_k}r$ for $k \in \{4,5\}$. Hence, $\mathsf{d}_R(j_k) = v_1 (v_1^{-1}v_k) \mathsf{d}'_R(j_k) = v_1 q^{n_k} \mathsf{d}'_R(j_k) = v_1 q^n r$. In particular, condition \ref{lem:redux_if:alpha_prime} of Lemma \ref{lem:redux_if} holds for $\{D(j_4), D(j_5)\}$, contradicting the irreducibility of $\cD'$. This completes the proof.
\end{proof}

\begin{lem}\label{claim:regularidad}
Let $p \in \{1,\ldots,s\}$ be such that $\#\cD' (p) \geq 2$ and let $D(p') = \max \cD' (p)$, $D(p'') = \max \cD'(p)\setminus\{D(p')\}$. Then, there exist $w \in \cW$ and $w' \leq_p \mathsf{\tilde{d}}(s)\mathsf{\tilde{d}}(p'')^{-1}$ such that $w$ and $w'$ are suffix dependent, $|w| \geq |\mathsf{\tilde{d}}(p')|$ and $|w'| > |\mathsf{\tilde{d}}(s)| - |\mathsf{\tilde{d}}(p)|$.
\end{lem}
\begin{proof}
Note that $p'' < p' < p$. 
Before proving the main statement of the lemma, we highlight two useful relations. First, note that
\begin{equation}\label{eq:regularidad:2}
\mathsf{d}_L(p'') \mathsf{d}_M(p'') \mathsf{\tilde{d}}(p'') = \mathsf{d}'_L(p'')
\end{equation}
as $D(p'')$ is simple. Second, since $\mathsf{u}_R$ and $\mathsf{u}'_L$ are, by Definition \ref{def:DUset}, the shortest words in $\cW$ satisfying $\mathsf{d}_R(p'')a(p'') \leq_p \mathsf{u}_R$ and $\mathsf{d}'_L(p'') \leq_s \mathsf{u}'_L$, respectively, we have, by condition \eqref{defi:double_interpretation:2} of the definition of simple d.i., that $|\mathsf{d}'_L(p'')| \geq \min(|\mathsf{u}_R|, |\mathsf{u}'_L|) \geq |\mathsf{\tilde{d}}(k)|$ for $k \in \{1,\dots,s\}$. This and the fact that $\mathsf{d}'_L(p'')$ and $\mathsf{\tilde{d}}(k)$ are both suffix of $\mathsf{u}'_L$ imply
\begin{equation}\label{eq:regularidad:3}
\text{$\mathsf{\tilde{d}}(k) \leq_s \mathsf{d}'_L(p'')$ for $k \in \{1,\dots,s\}$.}
\end{equation}

Now we are ready to prove the main statement of the lemma. Using \eqref{eq:regularidad:3} and $\mathsf{\tilde{d}}(p') \leq_p \mathsf{\tilde{d}}(p)$, we have $(\mathsf{d}'_L(p'') \mathsf{\tilde{d}}(p)^{-1}) \mathsf{\tilde{d}}(p') \leq_p \mathsf{d}'_L(p'')$. In addition, $\mathsf{d}_L(p'') \leq_p \mathsf{d}'_L(p'')$ by the simplicity of $D(p'')$. Thus, $(\mathsf{d}'_L(p'') \mathsf{\tilde{d}}(p)^{-1}) \mathsf{\tilde{d}}(p')$ and $\mathsf{d}_L(p'')$ are prefix dependent. In what follows, we split the proof in two cases according to which of these words is prefix of the other.
\begin{enumerate}[wide, labelwidth=!, labelindent=0pt]
\item[\textbf{(a)}] $(\mathsf{d}'_L(p'') \mathsf{\tilde{d}}(p)^{-1}) \mathsf{\tilde{d}}(p') \leq_p \mathsf{d}_L(p'')$. Observe that $\mathsf{\tilde{d}}(s) \leq_s \mathsf{u}'_L$ and $\mathsf{d}_M(p'') \mathsf{\tilde{d}}(p'') \leq_s \mathsf{d}'_L(p'') \leq_s \mathsf{u}'_L$, so $\mathsf{\tilde{d}}(s)$ and $\mathsf{d}_M(p'')\mathsf{\tilde{d}}(p'')$ are suffix dependent. In addition, from \eqref{eq:regularidad:2} and (a) we get
\begin{align}\label{eq:regularidad:x}
|\mathsf{d}_M(p'') \mathsf{\tilde{d}}(p'')| &=
|\mathsf{d}'_L(p'')|- |\mathsf{d}_L(p'')| \\ &\leq
|\mathsf{d}'_L(p'')|- |(\mathsf{d}'_L(p'') \mathsf{\tilde{d}}(p)^{-1})\mathsf{\tilde{d}}(p')| \notag \\ &=
|\mathsf{\tilde{d}}(p)| - |\mathsf{\tilde{d}}(p')| \leq 
|\mathsf{\tilde{d}}(s)|. \notag
\end{align}
We conclude that
\begin{equation*}
\mathsf{d}_M(p'')\mathsf{\tilde{d}}(p'') \leq_s \mathsf{\tilde{d}}(s).
\end{equation*}
Thus, it makes sense to define $w' \coloneqq \mathsf{\tilde{d}}(s) (\mathsf{d}_M(p'')\mathsf{\tilde{d}}(p''))^{-1}$. Clearly, $w' \leq_p \mathsf{\tilde{d}}(s) \mathsf{\tilde{d}}(p'')^{-1}$. Let $w \in \cW$ be a word satisfying $\mathsf{d}_L(p'') \leq_s w$, as in the definition of interpretation. Observe that, by \eqref{eq:regularidad:3} and \eqref{eq:regularidad:2},
\begin{equation*}
w' \leq_s \mathsf{d}'_L(p'')(\mathsf{d}_M(p'')\mathsf{\tilde{d}}(p''))^{-1} = \mathsf{d}_L(p'') 
\leq_s  w,
\end{equation*}
so $w$ and $w'$ are suffix dependent. It left to prove that $|w'| \geq |\mathsf{\tilde{d}}(s)| - |\mathsf{\tilde{d}}(p)|$ and $|w| \geq |\mathsf{\tilde{d}}(p')|$. For this, we note that in \eqref{eq:regularidad:x} it was shown that $|\mathsf{d}_M(p'')\mathsf{\tilde{d}}(p'')| \leq |\mathsf{\tilde{d}}(p)| - |\mathsf{\tilde{d}}(p')|$.
Thus,
\begin{equation*}
|w'| \geq |\mathsf{\tilde{d}}(s)| - |\mathsf{\tilde{d}}(p)| + |\mathsf{\tilde{d}}(p')| \geq \max(|\mathsf{\tilde{d}}(s)| - |\mathsf{\tilde{d}}(p)|, |\mathsf{\tilde{d}}(p')|).
\end{equation*}
We conclude that $|w'| \geq |\mathsf{\tilde{d}}(s)| - |\mathsf{\tilde{d}}(p)|$ and, since $w' \leq_s w$, $|w| \geq |w'| \geq |\mathsf{\tilde{d}}(p')|$. This completes the proof in case (a).

\medskip
\item[\textbf{(b)}] $\mathsf{d}_L(p'') <_p (\mathsf{d}'_L(p'') \mathsf{\tilde{d}}(p)^{-1}) \mathsf{\tilde{d}}(p')$.
We start by claiming that
\begin{equation}\label{eq:regularidad:1}
|\mathsf{\tilde{d}}(p'')| +|\mathsf{\tilde{d}}(p')| < |\mathsf{\tilde{d}}(p)|.
\end{equation}
Assume that \eqref{eq:regularidad:1} does not hold. Let $q$ be the shortest word satisfying $\mathsf{\tilde{d}}(p) \leq_s \prescript{\infty}{}{q}$. As we commented before Lemma \ref{claim:cota_sincronos}, condition $p',p'' < p$ implies that $\mathsf{\tilde{d}}(p')$, as well as $\mathsf{\tilde{d}}(p'')$, are prefixes and suffixes of $\mathsf{\tilde{d}}(p)$. So $|\mathsf{\tilde{d}}(p)| -|\mathsf{\tilde{d}}(p')|$ and $|\mathsf{\tilde{d}}(p)| -|\mathsf{\tilde{d}}(p'')|$ are periods of $\mathsf{\tilde{d}}(p)$. Moreover, since we are assuming \eqref{eq:regularidad:1} is not true, we also have that $(|\mathsf{\tilde{d}}(p)|-|\mathsf{\tilde{d}}(p')|)+(|\mathsf{\tilde{d}}(p)|-|\mathsf{\tilde{d}}(p'')|) \leq |\mathsf{\tilde{d}}(p)|$. Then, by Lemma \ref{lem:fine_wilf}, we obtain that $|q|$ divides $|\mathsf{\tilde{d}}(p)|-|\mathsf{\tilde{d}}(p')|$ and $|\mathsf{\tilde{d}}(p)|-|\mathsf{\tilde{d}}(p'')|$. Hence, there exists $n',n'' \in \N$ such that $q^{n'} = \mathsf{\tilde{d}}(p')^{-1}\mathsf{\tilde{d}}(p)$ and $q^{n''} = \mathsf{\tilde{d}}(p'')^{-1}\mathsf{\tilde{d}}(p)$. Now, since $p',p'' \in \cD'(p)$, we can write $\mathsf{d}'_M(p') \mathsf{d}'_R(p') a(p') = \mathsf{\tilde{d}}(p')^{-1} \mathsf{d}_R(p') a(p') \leq_p \mathsf{\tilde{d}}(p')^{-1} \mathsf{\tilde{d}}(p) = q^{n'} \leq_p q^\infty$ and, in a similar way, $\mathsf{d}'_M(p'') \mathsf{d}'_R(p'') a(p'') \leq_p q^\infty$. Thus, $\{D(p'),D(p'')\}$ is reducible by part \ref{lem:redux_if:prefix} of Lemma \ref{lem:redux_if}, which contradicts the irreducibility of $\cD'$. This proves the claim.
\smallskip

From \eqref{eq:regularidad:1} and \eqref{eq:regularidad:2} we get  
\begin{align*}
|(\mathsf{d}'_L(p'') \mathsf{\tilde{d}}(p)^{-1}) \mathsf{\tilde{d}}(p')| &=
|\mathsf{d}'_L(p'')|-|\mathsf{\tilde{d}}(p)|+|\mathsf{\tilde{d}}(p')| \\ &<
|\mathsf{d}'_L(p'')| - |\mathsf{\tilde{d}}(p'')| = 
|\mathsf{\tilde{d}}(p'')^{-1} \mathsf{d}'_L(p'')| =
|\mathsf{d}_L(p'') \mathsf{d}_M(p'')|.
\end{align*}
Then, since 
\begin{equation*}
(\mathsf{d}'_L(p'') \mathsf{\tilde{d}}(p)^{-1}) \mathsf{\tilde{d}}(p') \leq_p
(\mathsf{d}'_L(p'') \mathsf{\tilde{d}}(p)^{-1}) \mathsf{\tilde{d}}(p) =
\mathsf{d}'_L(p'') = \mathsf{d}_L(p'') \mathsf{d}_M(p'')\mathsf{\tilde{d}}(p''),
\end{equation*}
we obtain that $(\mathsf{d}'_L(p'') \mathsf{\tilde{d}}(p)^{-1}) \mathsf{\tilde{d}}(p') <_p \mathsf{d}_L(p'') \mathsf{d}_M(p'')$. This and (b) can be written together as
\begin{equation*}
\mathsf{d}_L(p'') <_p
(\mathsf{d}'_L(p'') \mathsf{\tilde{d}}(p)^{-1}) \mathsf{\tilde{d}}(p') <_p 
\mathsf{d}_L(p'') \mathsf{d}_M(p'').
\end{equation*}
Thus, we can write $\mathsf{d}_L(p'') \mathsf{d}_M(p'') = vwv'$, where $v \in \mathsf{d}_L(p'') \cW^*$, $w \in \cW$, $v' \in \cW^*$ and
\begin{equation}\label{eq:regularidad:4}
v <_p (\mathsf{d}'_L(p'') \mathsf{\tilde{d}}(p)^{-1}) \mathsf{\tilde{d}}(p') \leq_p vw.
\end{equation}
The word $w$ is the one we need in the statement of the lemma. To define $w'$, we first note that $\mathsf{\tilde{d}}(s) \leq_s \mathsf{d}'_L(p'')$ and $v' \mathsf{\tilde{d}}(p'') \leq_s \mathsf{d}_L(p'') \mathsf{d}_M(p'') \mathsf{\tilde{d}}(p'') = \mathsf{d}'_L(p'')$, so $\mathsf{\tilde{d}}(s)$ and $v' \mathsf{\tilde{d}}(p'')$ are suffix dependent. Moreover, using \eqref{eq:regularidad:4} we get
\begin{equation}\label{eq:regularidad:4.5}
|v' \mathsf{\tilde{d}}(p'')| = |\mathsf{d}'_L(p'')| - |vw| \leq 
|\mathsf{d}'_L(p'')| - |(\mathsf{d}'_L(p'') \mathsf{\tilde{d}}(p)^{-1}) \mathsf{\tilde{d}}(p')| =
|\mathsf{\tilde{d}}(p)| - |\mathsf{\tilde{d}}(p')|.
\end{equation}
Then, $|v' \mathsf{\tilde{d}}(p'')| \leq |\mathsf{\tilde{d}}(p)| - |\mathsf{\tilde{d}}(p')| \leq |\mathsf{\tilde{d}}(s)|$ and $v'\mathsf{\tilde{d}}(p'') \leq_s \mathsf{\tilde{d}}(s)$. Now it makes sense to define $w' \coloneqq \mathsf{\tilde{d}}(s) (v'\mathsf{\tilde{d}}(p''))^{-1}$, which clearly verifies $w' \leq_p \mathsf{\tilde{d}}(s)\mathsf{\tilde{d}}(p'')^{-1}$. It is also clear that $w$ and $w'$ are suffix dependent. Indeed, from \eqref{eq:regularidad:3} and \eqref{eq:regularidad:2} we have $w' \leq_s \mathsf{d}'_L(p'')(v'\mathsf{\tilde{d}}(p''))^{-1} = vw$. 

Now, from \eqref{eq:regularidad:4.5}, $|w'| \geq |\mathsf{\tilde{d}}(s)|- |\mathsf{\tilde{d}}(p)| + |\mathsf{\tilde{d}}(p')| \geq |\mathsf{\tilde{d}}(s)| - |\mathsf{\tilde{d}}(p)|$, proving the desired condition on the length of $w'$. It only left to prove that $|w| \geq |\mathsf{\tilde{d}}(p')|$. We argue by contradiction. Assume that
\begin{equation}\label{eq:regularidad:5}
|w| < |\mathsf{\tilde{d}}(p')|.
\end{equation}
First, we prove that it makes sense to define the word
\begin{equation}\label{eq:regularidad:6}
w'' \coloneqq 
((\mathsf{d}'_L(p'')\mathsf{\tilde{d}}(p)^{-1})^{-1}v)^{-1}\mathsf{d}_R(p') \in \cA^+.
\end{equation}
From \eqref{eq:regularidad:4} and \eqref{eq:regularidad:5} we get $|v| \geq |(\mathsf{d}'_L(p'')\mathsf{\tilde{d}}(p)^{-1})\mathsf{\tilde{d}}(p')| -|w| > |\mathsf{d}'_L(p'')\mathsf{\tilde{d}}(p)^{-1}|$. But, $v \leq_p \mathsf{d}_L(p'')\mathsf{d}_M(p'') \leq_p \mathsf{d}'_L(p'')$ and $\mathsf{d}'_L(p'')\mathsf{\tilde{d}}(p)^{-1} \leq_p \mathsf{d}'_L(p'')$, so $\mathsf{d}'_L(p'')\mathsf{\tilde{d}}(p)^{-1} <_p v$ and $(\mathsf{d}'_L(p'')\mathsf{\tilde{d}}(p)^{-1})^{-1}v$ exists and is not the empty word. Hence, by \eqref{eq:regularidad:4}, 
\begin{equation}\label{eq:regularidad:middle}
(\mathsf{d}'_L(p'')\mathsf{\tilde{d}}(p)^{-1})^{-1}v <_p 
\mathsf{\tilde{d}}(p') \leq_p \mathsf{d}_R(p')
\end{equation}
and $w''$ is well defined.

Now, we have $vw \leq_p \mathsf{d}_L(p'')\mathsf{d}_M(p'') \leq_p \mathsf{d}'_L(p'')$ and, using $p'\in\cD'(p)$, that $(\mathsf{d}'_L(p'') \mathsf{\tilde{d}}(p)^{-1}) \mathsf{d}_R(p') \leq_p (\mathsf{d}'_L(p'') \mathsf{\tilde{d}}(p)^{-1}) \mathsf{\tilde{d}}(p) = \mathsf{d}'_L(p'')$. Thus, $vw$ and \break $(\mathsf{d}'_L(p'') \mathsf{\tilde{d}}(p)^{-1}) \mathsf{d}_R(p')$ are prefix dependent. Therefore, there are two cases: $vw$ is prefix of $(\mathsf{d}'_L(p'') \mathsf{\tilde{d}}(p)^{-1}) \mathsf{d}_R(p')$ and $(\mathsf{d}'_L(p'') \mathsf{\tilde{d}}(p)^{-1}) \mathsf{d}_R(p')$ is a strict prefix of $vw$; in each of these cases we will build a reduction for $D(p')$, producing a contradiction.

\medskip
\begin{enumerate}[wide, labelwidth=!, labelindent=0pt]

\item[\textbf{(b.1)}] $vw \leq_p (\mathsf{d}'_L(p'') \mathsf{\tilde{d}}(p)^{-1}) \mathsf{d}_R(p')$.
We start by building a d.i. of $w''$. Note that
\begin{equation}\label{eq:regularidad:7}
w'' a(p') \leq_p wv'\mathsf{\tilde{d}}(p'').
\end{equation}
Indeed, since $D(p') \in \cD'(p)$ and $(\mathsf{d}'_L(p'')\mathsf{\tilde{d}}(p)^{-1})\mathsf{\tilde{d}}(p) = \mathsf{d}'_L(p'') = vwv'\mathsf{\tilde{d}}(p'')$, we have $\mathsf{d}_R(p')a(p') \leq_p \mathsf{\tilde{d}}(p) = (\mathsf{d}'_L(p'')\mathsf{\tilde{d}}(p)^{-1})^{-1}vwv'\mathsf{\tilde{d}}(p'')$, which implies \eqref{eq:regularidad:7}. Now, since $w \in \cW$, $v' \in \cW^*$ and $\mathsf{\tilde{d}}(p') <_p \mathsf{u}_R$, the word $wv'\mathsf{\tilde{d}}(p')$ has an interpretation of the form $J = w, v', \mathsf{\tilde{d}}(p'), a$. Moreover, using (b.1) we can get $|w''| = |\mathsf{d}_R(p')|+|\mathsf{d}'_L(p'')\mathsf{\tilde{d}}(p)^{-1}|-|v| \geq |w|$. Hence, by \eqref{eq:regularidad:7}, Lemma \ref{lem:inherit} can be applied with $J$ to obtain an interpretation of $w''$ having the form $I' = w, r, r', a(p')$. We need another interpretation of $w''$. Note that in the middle step of \eqref{eq:regularidad:middle} we showed that $(\mathsf{d}'_L(p'')\mathsf{\tilde{d}}(p)^{-1})^{-1}v <_p \mathsf{\tilde{d}}(p')$. In particular, the word $((\mathsf{d}'_L(p'')\mathsf{\tilde{d}}(p)^{-1})^{-1}v)^{-1} \mathsf{\tilde{d}}(p')$ is nonempty and is a suffix of $\mathsf{u}'_L \in \cW$. Then,
\begin{equation*}
I \coloneqq ((\mathsf{d}'_L(p'')\mathsf{\tilde{d}}(p)^{-1})^{-1}v)^{-1} \mathsf{\tilde{d}}(p'), 
\mathsf{d}'_M(p'), \mathsf{d}'_R(p'), a'(p'))
\end{equation*}
is an interpretation of $w''$ (here, we used that $\mathsf{\tilde{d}}(p')\mathsf{d}'_M(p')\mathsf{d}'_R(p') = \mathsf{d}_R(p')$). We set $D = (I,I')$. Since $a(p') \not= a'(p')$, $D$ is a d.i. of $w''$.

Now we can conclude the proof of this case. From \eqref{eq:regularidad:4} we have $|v| \geq $ \break 
$|(\mathsf{d}'_L(p'')\mathsf{\tilde{d}}(p)^{-1})\mathsf{\tilde{d}}(p')| - |w|$, which implies $|((\mathsf{d}'_L(p'')\mathsf{\tilde{d}}(p)^{-1})^{-1}v)^{-1} \mathsf{\tilde{d}}(p')| \leq |w| \leq |wr|$. This and that $w \in \cW$ allow us to use Lemma \ref{lem:extract_di} to obtain a simple d.i. $E$ of a word $e$ such that $e \leq_s w''$. Since $w'' <_s \mathsf{d}_R(p') <_s d(p')$, we have that $D(p')$ reduces to $E$. This is the desired contradiction.

\medskip
\item[\textbf{(b.2)}] $(\mathsf{d}'_L(p'') \mathsf{\tilde{d}}(p)^{-1}) \mathsf{d}_R(p') <_p vw$.
We are going to build a simple d.i. $D = (I;I')$ of $\mathsf{d}_R(p') <_s d(p')$, proving, thereby, that $D(p')$ has a reduction.

Let $I' = \mathsf{\tilde{d}}(p'), \mathsf{d}'_M(p'), \mathsf{d}'_R(p'), a'(p')$. It is clear that $I'$ is an interpretation of $\mathsf{d}_R(p')$ since $\mathsf{\tilde{d}}(p') \leq_s \mathsf{u}'_L$, $\mathsf{d}'_M(p') \in \cW^*$, $\mathsf{d}'_R(p')a'(p') \leq_p \mathsf{u}'_R$ and $|\mathsf{\tilde{d}}(p')| > |\mathsf{\tilde{d}}(p'')| \geq 0$. To define $I$, observe that in the proof of \eqref{eq:regularidad:6} we showed that $(\mathsf{d}'_L(p'')\mathsf{\tilde{d}}(p)^{-1})^{-1}v$ exists and is not the empty word. But, moreover, from $v \in \mathsf{d}_L(p'') \cW^*$ we see that we can write $(\mathsf{d}'_L(p'')\mathsf{\tilde{d}}(p)^{-1})^{-1}v = rr'$ in such a way that $r$ is a nonempty suffix of some word in $\cW$ and $r' \in \cW^*$. Since, by definition, $\mathsf{d}_R(p') = rr'w''$, to prove that $I \coloneqq  r, r', w'', a(p')$ is an interpretation of $\mathsf{d}_R(p')$ it is enough to show that $w''a(p') \leq_p w$. From (b.2) we get $rr'w'' = \mathsf{d}_R(p') <_p rr'w$, so $w''a' \leq_p w$ for some $a' \in \cA$. Then, using that $vw \leq_p vwv'\mathsf{\tilde{d}}(p'') = \mathsf{d}'_L(p'')$, we obtain
\begin{align*}
\mathsf{d}_R(p')a' \leq_p rr'w &= 
(\mathsf{d}'_L(p'')\mathsf{\tilde{d}}(p)^{-1})^{-1}vw \\ &\leq_p
(\mathsf{d}'_L(p'')\mathsf{\tilde{d}}(p)^{-1})^{-1}
\mathsf{d}'_L(p'') =
\mathsf{\tilde{d}}(p) \leq_p \mathsf{u}_R.
\end{align*}
Since we also have $\mathsf{d}_R(p')a(p') \leq_p \mathsf{u}_R$, we deduce that $a' = a(p')$. Hence, $w''a(p') \leq_p w$ and $I$ is an interpretation of $\mathsf{d}_R(p')$. Being $a(p') \not= a'(p')$, we conclude that $D \coloneqq (I;I')$ is a d.i. of $\mathsf{d}_R(p')$. 

Finally, we prove that $D$ is simple. Using the middle step in \eqref{eq:regularidad:middle} we get $rr' = (\mathsf{d}'_L(p'')^{-1}\mathsf{\tilde{d}}(p))^{-1}v <_p \mathsf{\tilde{d}}(p')$. This implies that $\mathsf{d}'_M(p')\mathsf{d}'_R(p') = \mathsf{\tilde{d}}(p')^{-1}\mathsf{d}_R(p') \leq_s (rr')^{-1}\mathsf{d}_R(p') = w''$, which is the first condition in Definition \ref{defi:double_interpretation}. Since $w''a(p') \leq_p w$ and, by \eqref{eq:regularidad:5}, $|\mathsf{\tilde{d}}(p')| \geq |w|$, the second condition also holds. Hence, $D$ is simple and $D(p')$ reduces to it.
\end{enumerate}
\end{enumerate}
\end{proof}

Remark that in the last paragraph it was the first time that in a proof we build a reduction to a simple d.i. satisfying the \emph{second} condition of \eqref{defi:double_interpretation:2} in Definition \ref{defi:double_interpretation}.

\subsection{Proof of Proposition \ref{lem:bound_on_irreducible}}

\begin{cprop}[\ref{lem:bound_on_irreducible}]
Any irreducible subset of  $\cD_U$ has at most $61(\#\cW)$ elements.
\end{cprop}
\begin{proof}
Let $\cD'$ be an irreducible subset of $\cD_U$. Recall that, with the notation introduced above, $D(1) < \dots < D(s)$ are the elements of $\cD'$ deployed in increasing order, $\cD' (s+1) = \cD' $  and $\cD' (j) =
\{D \in \cD' : \mathsf{d}_Ra_D \leq_p \mathsf{\tilde{d}}(j)\} =
\{D \in \cD' : |\mathsf{d}_R| < |\mathsf{\tilde{d}}(j)|\}$
for $j \in \{1,\dots,s\}$. 

We define recursively a finite decreasing sequence $(p_i)_{i=0}^{t+1}$. We start with $p_0 = s+1$. Then, for $i\geq 0$: a) if 
$\#\cD' (p_i) \leq 1$ we put $p_{i+1} = 1$ and the procedure stops; b) if $\#\cD' (p_i) > 1$, set $D(p_{i+1}) = \max \cD'(p_i)$.
Observe that $\cD'(p_{i+1}) \subsetneq \cD'(p_i)$. Let $t \geq 0$ be the first integer for which $\# \cD'(p_t) \leq 1$, so that  $\cD'(p_{t+1})=\cD'(1) = \emptyset$. This construction gives
\begin{equation*}
\cD'  =
\bigcup_{i = 0}^{t} \cD' (p_i) \backslash \cD' (p_{i+1}).
\end{equation*}
From Lemma \ref{claim:cota_sincronos} we get that $\#\cD'  \leq 6t+1$. To complete the proof we are going to show that $t \leq 8\# \cW+2$.

We proceed by contradiction, so we suppose $t > 8 \#\cW+2$. This will imply that $\cD'$ is reducible, which contradicts our hypothesis.

Let $1 \leq i \leq t-1$. Since $p_i \not= s+1$ and $\#\cD'(p_i) > 1$, we can define $D(p''_i) = \max\cD'(p_i)\setminus\{D(p_{i+1})\}$ and use Lemma \ref{claim:regularidad} with $\cD'(p_i)$ to obtain suffix dependent words $w_i \in  \cW$ and $w'_i \in \cA^*$ such that
\begin{equation}\label{eq:lem:bound_on_irreducible:1}
\text{(i) $|w_i| > |\mathsf{\tilde{d}}(p_{i+1})|$,\quad
(ii) $|w'_i| \geq |\mathsf{\tilde{d}}(s)|-|\mathsf{\tilde{d}}(p_i)|$, \quad
(iii) $w'_i \leq_p \mathsf{\tilde{d}}(s)\mathsf{\tilde{d}}(p''_i)^{-1}$.}
\end{equation}
Then, by the Pigeonhole Principle, we can find $1 \leq i_5 < \dots < i_1 \leq t-1$ such that
\begin{equation*}
\text{(a) $w \coloneqq w_{i_1} = \dots = w_{i_5}$  and
(b) $i_{k+1} +2 \leq i_k$ for any $k \in \{1,\dots,4\}$.}
\end{equation*}
Using (a) and (b) we are going to obtain relations \eqref{eq:lem:bound_on_irreducible:2} and \eqref{eq:lem:bound_on_irreducible:2.1} below.

\smallskip
First, we use (b) to prove that
\begin{equation}\label{eq:lem:bound_on_irreducible:2}
\mathsf{\tilde{d}}(s)\mathsf{\tilde{d}}(p_{i_{k+1}})^{-1} <_p w'_{i_{k+1}} \leq_p
\mathsf{\tilde{d}}(s)\mathsf{\tilde{d}}(p_{i_k})^{-1} <_p w'_{i_k}
\quad\text{for any $k \in \{1,\dots,4\}$.}
\end{equation}
Let $k \in \{1,\dots,4\}$. By (b), we have $i_{k+1} \leq i_{k+1}+1 < i_{k+1}+2 \leq t-1$. Thus, $D(p_{i_{k+1}+2}) < D(p_{i_{k+1}+1})$ and $D(p_{i_{k+1}+1}), D(p_{i_{k+1}+2}) \in \cD'(p_{i_{k+1}})$, which implies that $p''_{i_{k+1}} \geq p_{i_{k+1}+2}$ by the definition of $p''_{i_{k+1}}$. Being $p_{i_{k+1}+2} \geq p_{i_k}$ by (b), we obtain $p''_{i_{k+1}} \geq p_{i_k}$. This and (iii) of \eqref{eq:lem:bound_on_irreducible:1} imply $w'_{i_{k+1}} \leq_p \mathsf{\tilde{d}}(s) \mathsf{\tilde{d}}(p''_{i_{k+1}})^{-1} \leq_p \mathsf{\tilde{d}}(s) \mathsf{\tilde{d}}(p_{i_k})^{-1}$. This proves the middle inequality of \eqref{eq:lem:bound_on_irreducible:2}.
Let $k \in \{1,\dots,5\}$. Since $w'_{i_k} \leq_p \mathsf{\tilde{d}}(s)\mathsf{\tilde{d}}(p''_{i_k})^{-1} \leq_p \mathsf{\tilde{d}}(s)$ by (iii) of \eqref{eq:lem:bound_on_irreducible:1} and $\mathsf{\tilde{d}}(s)\mathsf{\tilde{d}}(p_{i_k})^{-1} \leq_p \mathsf{\tilde{d}}(s)$, we have that $w'_{i_k}$ and $\mathsf{\tilde{d}}(s)\mathsf{\tilde{d}}(p_{i_k})^{-1}$ are prefix dependent. Moreover, $|w'_{i_k}| > |\mathsf{\tilde{d}}(s)\mathsf{\tilde{d}}(p_{i_k})^{-1}|$ by (ii) of \eqref{eq:lem:bound_on_irreducible:1}, so $\mathsf{\tilde{d}}(s)\mathsf{\tilde{d}}(p_{i_k})^{-1} <_p w'_{i_k}$. This proves the first and last inequality of \eqref{eq:lem:bound_on_irreducible:2}, completing the proof.

\smallskip
Thanks to \eqref{eq:lem:bound_on_irreducible:2}, the word $(\mathsf{\tilde{d}}(s) \mathsf{\tilde{d}}(p_{i_k})^{-1})^{-1}w'_{i_{k'}}$ exists for any $1\leq k' \leq k \leq 5$. We will use this fact freely through the proof.

Next, we want to obtain from (a) that
\begin{equation}\label{eq:lem:bound_on_irreducible:2.1}
\text{$(\mathsf{\tilde{d}}(s) \mathsf{\tilde{d}}(p_{i_4})^{-1})^{-1}w'_{i_k}
\leq_s w$ for $k \in \{1,\dots,4\}$.}
\end{equation}
By (a) and (i) of \eqref{eq:lem:bound_on_irreducible:1}, we have $|\mathsf{\tilde{d}}(p_{i_4})| \leq |\mathsf{\tilde{d}}(p_{i_5+1})| \leq |w|$. This and (iii) imply
\begin{align*}
|(\mathsf{\tilde{d}}(s) \mathsf{\tilde{d}}(p_{i_4})^{-1})^{-1}w'_{i_k}| \leq |\mathsf{\tilde{d}}(s)\mathsf{\tilde{d}}(p''_{i_k})^{-1}| - |\mathsf{\tilde{d}}(s)\mathsf{\tilde{d}}(p_{i_4})^{-1}| \leq |\mathsf{\tilde{d}}(p_{i_4})| \leq |w|.
\end{align*}
But, being $w$ and $(\mathsf{\tilde{d}}(s) \mathsf{\tilde{d}}(p_{i_4})^{-1})^{-1}w'_{i_k}$ suffix dependent since $w$ and $w'_{i_k}$ have the same property and $(\mathsf{\tilde{d}}(s) \mathsf{\tilde{d}}(p_{i_4})^{-1})^{-1}w'_{i_k} \leq_s w'_{i_k}$, we obtain that $(\mathsf{\tilde{d}}(s) \mathsf{\tilde{d}}(p_{i_4})^{-1})^{-1}w'_{i_k}
\leq_s w$, as desired.

\smallskip
Now we use relations \eqref{eq:lem:bound_on_irreducible:2} and \eqref{eq:lem:bound_on_irreducible:2.1} to obtain restrictions on the smallest period of $v \coloneqq (\mathsf{\tilde{d}}(s) \mathsf{\tilde{d}}(p_{i_4})^{-1})^{-1}w'_{i_1}$. More precisely, we claim that if $q \in \cA^+$ is the shortest word satisfying $v \leq_p q^\infty$, then $|q|$ divides $|\mathsf{\tilde{d}}(p_{i_4})|-|\mathsf{\tilde{d}}(p_{i_k})|$ for $k \in \{2,3\}$.

Fix $k \in \{2,3\}$. First, observe that $v \leq_s w$ and $v((w'_{i_2})^{-1}w'_{i_1})^{-1} =  \break (\mathsf{\tilde{d}}(s) \mathsf{\tilde{d}}(p_{i_4})^{-1})^{-1}w'_{i_2} \leq_s w$ by \eqref{eq:lem:bound_on_irreducible:2.1}. Being $(w'_{i_2})^{-1}w'_{i_1} \not= 1$ by \eqref{eq:lem:bound_on_irreducible:1}, we deduce that $v \leq_s \prescript{\infty}{}{((w'_{i_2})^{-1}w'_{i_1})}$. This implies that $|q| \leq |(w'_{i_2})^{-1}w'_{i_1}|$. Thus,
\begin{align}\label{eq:lem:bound_on_irreducible:inter1}
|q| + |\mathsf{\tilde{d}}(p_{i_4})\mathsf{\tilde{d}}(p_{i_k})^{-1}| &\leq
|(w'_{i_2})^{-1}w'_{i_1}| +
|\mathsf{\tilde{d}}(p_{i_4})\mathsf{\tilde{d}}(p_{i_k})^{-1}| \\ &=
|v| + |(\mathsf{\tilde{d}}(s)\mathsf{\tilde{d}}(p_{i_k})^{-1})^{-1}w'_{i_2}| \leq |v|, \notag
\end{align}
where $(\mathsf{\tilde{d}}(s)\mathsf{\tilde{d}}(p_{i_k})^{-1})^{-1}w'_{i_2}$ exists because $k \geq 2$.

Second, since $w'_{i_1} \leq_p \mathsf{\tilde{d}}(s)$ by (iii) of \eqref{eq:lem:bound_on_irreducible:1}, we have that $v = (\mathsf{\tilde{d}}(s) \mathsf{\tilde{d}}(p_{i_4})^{-1})^{-1}w'_{i_1} \leq_p \mathsf{\tilde{d}}(p_{i_4}) \leq_p \mathsf{u}_R$ and $(\mathsf{\tilde{d}}(s) \mathsf{\tilde{d}}(p_{i_k})^{-1})^{-1}w'_{i_1} \leq_p \mathsf{\tilde{d}}(p_{i_k}) \leq_p \mathsf{u}_R$.
Therefore,
\begin{equation*}
\text{$v \leq_p \mathsf{u}_R$ and 
$(\mathsf{\tilde{d}}(p_{i_4})\mathsf{\tilde{d}}(p_{i_k})^{-1})^{-1}v = (\mathsf{\tilde{d}}(s) \mathsf{\tilde{d}}(p_{i_k})^{-1})^{-1}w'_{i_1} \leq_p \mathsf{u}_R$.}
\end{equation*}
This and the fact that, by \eqref{eq:lem:bound_on_irreducible:1}, $(\mathsf{\tilde{d}}(p_{i_4})\mathsf{\tilde{d}}(p_{i_k})^{-1}) \not= 1$ imply that $v \leq_p (\mathsf{\tilde{d}}(p_{i_4})\mathsf{\tilde{d}}(p_{i_k})^{-1})^\infty$. Hence,
\begin{equation}\label{eq:lem:bound_on_irreducible:inter2}
\text{$|\mathsf{\tilde{d}}(p_{i_4})\mathsf{\tilde{d}}(p_{i_k})^{-1}|$ is a period of $v$.}
\end{equation}
Then, from \eqref{eq:lem:bound_on_irreducible:inter1} and \eqref{eq:lem:bound_on_irreducible:inter2}, we can use Lemma \ref{lem:fine_wilf} with $v$ to deduce that $|q|$ divides $|\mathsf{\tilde{d}}(p_{i_4})\mathsf{\tilde{d}}(p_{i_k})^{-1}|$, proving the claim.

Let now $\tilde{q} \in \cA^+$ be the shortest word such that $\mathsf{\tilde{d}}(p_{i_4}) \leq_p \tilde{q}^\infty$. From the last claim, we have for $k \in \{2,3\}$ that $\mathsf{\tilde{d}}(p_{i_4})\mathsf{\tilde{d}}(p_{i_k})^{-1} = q^{n_k}$ for some $n_k \geq 1$. 
Then, since $|\mathsf{\tilde{d}}(p_{i_4})\mathsf{\tilde{d}}(p_{i_k})^{-1}|$ is a period of $\mathsf{\tilde{d}}(p_{i_4})$ as $p_{i_k} < p_{i_4}$, we obtain $\mathsf{\tilde{d}}(p_{i_4}) \leq_p (\mathsf{\tilde{d}}(p_{i_4})\mathsf{\tilde{d}}(p_{i_k})^{-1})^\infty = q^\infty$ and  $\tilde{q} \leq_p q$. 
Since, $v \leq_p \mathsf{\tilde{d}}(p_{i_4}) \leq_p \tilde{q}^\infty$, we also have $q \leq_p \tilde{q}$. Therefore, $\tilde{q} = q$.

Now we can finish the proof of the proposition. 
Since $\mathsf{\tilde{d}}(p_{i_4}) \leq_p q^\infty$, there are $n\geq0$ and $r <_p q$ such that $\mathsf{\tilde{d}}(p_{i_4}) = q^nr$. Then, for $k \in \{2,3\}$, we have $\mathsf{\tilde{d}}(p_{i_k}) = q^{-n_k}\mathsf{\tilde{d}}(p_{i_4}) = q^{n-n_k}r$. Being $p_{i_2}, p_{i_3} \in \cD'(p_{i_4})$, we get
$$ \mathsf{\tilde{d}}'_M(p_{i_k}) \mathsf{\tilde{d}}'_R(p_{i_k}) a(p_{i_k}) =
\mathsf{\tilde{d}}(p_{i_k})^{-1} \mathsf{\tilde{d}}_R(p_{i_k})a(p_{i_k}) \leq_p
\mathsf{\tilde{d}}(p_{i_k})^{-1}\mathsf{\tilde{d}}(p_{i_4}) = r^{-1}q^{n_k} \leq_p r^{-1}q^\infty.$$
Thus, condition \ref{lem:redux_if:prefix} of Lemma \ref{lem:redux_if} holds, which implies that $\{D(p_{i_2}), D(p_{i_3})\}$ is reducible, contradicting our hypothesis.

\end{proof}

\section*{Acknowledgement}
The authors thank the anonymous referee for his/her careful reading of the manuscript and useful suggestions.

\bibliography{biblio}
\bibliographystyle{amsalpha}

\end{document}